\documentclass[12pt,a4paper]{amsart}

\usepackage{amsmath}
\usepackage{amscd}
\usepackage{amssymb}
\numberwithin{equation}{section}

\input{prepictex}
\input{pictex}
\input{postpictex}
\chardef\bslash=`\\ 

\makeatletter
\def\verbatim{\interlinepenalty\@M \@verbatim
  \leftskip\@totalleftmargin\advance\leftskip2pc
  \frenchspacing\@vobeyspaces \@xverbatim}
\makeatother
\newtheorem{theorem}{Theorem}[section]
\newtheorem{corollary}[theorem]{Corollary}
\newtheorem{lemma}[theorem]{Lemma}
\newtheorem{proposition}[theorem]{Proposition}

\theoremstyle{definition}
\newtheorem{definition}[theorem]{Definition}
\newtheorem{remark}[theorem]{Remark}

\newcounter{picture}

\DeclareMathOperator{\conv}{conv}
\newcommand{\CC}{{\mathbb C}}

\newcommand{\KK}{{\mathbb K}}
\newcommand{\NN}{{\mathbb N}}

\newcommand{\QQ}{{\mathbb Q}}

\newcommand{\TT}{{\mathbb T}}
\newcommand{\ZZ}{{\mathbb Z}}
\newcommand{\cA}{{\mathcal A}}
\newcommand{\cB}{{\mathcal B}}
\newcommand{\cC}{{\mathcal C}}
\newcommand{\cE}{{\mathcal E}}
\newcommand{\cF}{{\mathcal F}}
\newcommand{\cH}{{\mathcal H}}
\newcommand{\cK}{{\mathcal K}}

\newcommand{\cO}{{\mathcal O}}

\newcommand{\D}{{\Delta}}
\newcommand{\dd}{{{\delta}'}}
\newcommand{\G}{{\Gamma}}
\newcommand{\Om}{{\Omega}}
\newcommand{\g}{{\gamma}}
\newcommand{\s}{{\sigma}}

\newcommand{\fa}{{\mathfrak a}}
\newcommand{\fp}{{\mathfrak p}}
\newcommand{\fs}{{\mathfrak s}}
\newcommand{\ft}{{\mathfrak t}}

\newcommand{\fG}{{\mathfrak G}}
\newcommand{\fP}{{\mathfrak P}}
\newcommand{\fS}{{\mathfrak S}}
\newcommand{\fT}{{\mathfrak T}}
\newcommand{\fW}{{\mathfrak W}}

\newcommand{\ofW}{{\overline{\mathfrak W}}}
\newcommand{\oa}{{\overline{\alpha}}}
\newcommand{\ou}{{\overline u}}
\newcommand{\ov}{{\overline v}}
\newcommand{\ow}{{\overline w}}
\newcommand{\oW}{{\overline W}}
\newcommand{\zomatrix}{{$\{0,1\}$-matrix}}
\newcommand{\zomatrices}{{$\{0,1\}$-matrices}}
\newcommand{\Ind}{{\bf 1}} 
\newcommand{\PGL}{{\text{\rm{PGL}}}}
\newcommand{\SL}{{\text{\rm{SL}}}}
\newcommand{\wt}{\widetilde}

\begin{document}

\title[]{Affine buildings, tiling systems and higher rank Cuntz-Krieger algebras}

\date{February 2, 1999}
\author{Guyan Robertson }
\address{Mathematics Department, University of Newcastle, Callaghan, NSW
2308, Australia}
\email{guyan@maths.newcastle.edu.au}
\author{Tim Steger}
\address{Istituto Di Matematica e Fisica, Universit\`a degli Studi di
Sassari, Via Vienna 2, 07100 Sassari, Italia}
\email{steger@ssmain.uniss.it}
\subjclass{Primary 46L35; secondary 46L55, 22D25, 51E24.}
\keywords{$C^*$-algebra, subshift of finite type, affine building}
\thanks{This research was supported by the Australian Research Council.}
\thanks{ \hfill Typeset by  \AmS-\LaTeX}

\begin{abstract}
To an $r$-dimensional subshift of finite type satisfying certain special properties
we associate a $C^*$-algebra $\cA$. This algebra
is a higher rank version of a Cuntz-Krieger algebra. In particular, it
is simple, purely infinite and nuclear. We study an example: if $\G$ is a group
acting freely on the vertices of an $\wt A_2$ building, with finitely many orbits, and
if $\Omega$ is the boundary of that building, then $C(\Om)\rtimes \G$ is the algebra
associated to a certain two dimensional subshift.
\end{abstract}

\maketitle
\section*{Introduction}

This paper falls into two parts. The self-contained first part develops
the theory of a class of $C^*$-algebras which are higher rank generalizations of the  Cuntz-Krieger
algebras \cite{ck,c,c'}. We start with a set of $r$-dimensional words, based on an alphabet $A$,
we define transition matrices $M_j$ in each of $r$ directions, satisfying certain
conditions (H0)-(H3). The $C^*$-algebra $\cA$ is then the unique $C^*$-algebra
generated by a family of partial isometries $s_{u,v}$
indexed by compatible $r$-dimensional words $u,v$ and satisfying
relations (\ref{rel1*}) below. If $r=1$ then $\cA$ is a Cuntz--Krieger algebra.
We prove that the algebra $\cA$ is simple, purely infinite and stably isomorphic to
the crossed product of an AF-algebra by a $\ZZ^r$-action.

The last part of the paper (Section~\ref{boundary-algebra}) studies in detail one particularly
interesting example. This example was the authors' motivation for introducing these algebras.
Let $\cB$ be an affine building of type~$\wt A_2$.
Let $\G$ be a group of type rotating automorphisms of $\cB$ which acts freely on the vertex set
with finitely many orbits.  There is a natural action of $\G$ on the boundary $\Omega$
of $\cB$, and we can form the universal crossed product algebra $C(\Om)\rtimes \G$.
This algebra is isomorphic to an algebra of the form $\cA$ obtained by the preceding construction.
In the case where $\G$ also acts transitively on the vertices of $\cB$
the algebra $C(\Om)\rtimes \G$ was previously studied in \cite{rs}, where simplicity was proved.
In a sequel to this paper \cite{rs'} we explicitly compute the K-theory of some of these algebras.

In \cite{s} \textsc{J. Spielberg} treated an analogous example in rank $1$: $\G$ is a free group,
the building is a tree, and $C(\Om)\rtimes \G$ is isomorphic to an ordinary Cuntz-Krieger algebra.
In fact \cite{s} deals more generally with the case where $\G$ is a free product of cyclic groups.
The generalizations in this paper are motivated by Spielberg's work.

We now introduce some basic notation and terminology.
Let $\ZZ_+$ denote the set of nonnegative integers.
Let $[m,n]$ denote $\{m,m+1, \dots , n\}$, where $m \le n$ are integers. If $m,n \in \ZZ^r$,
say that $m \le n$ if $m_j \le n_j$ for $1 \le j \le r$, and when $m \le n$, let
$[m,n] = [m_1,n_1] \times \dots \times [m_r,n_r]$. In $\ZZ^r$, let $0$ denote the zero vector
and let $e_j$ denote the $j^{th}$ standard unit basis vector.  We fix a finite set $A$ (an ``alphabet'').

A \emph{$\{0,1\}$-matrix} is a matrix with entries in $\{0,1\}$.
Choose nonzero  $\{0,1\}$-matrices  $M_1, M_2,\dots, M_r$ and denote their elements by
 $M_j(b,a) \in \{0,1\}$ for $a,b \in A$.
If $m,n \in \ZZ^r$ with $m \le n$, let
$$W_{[m,n]} = \{ w: [m,n] \to A ;\ M_j(w(l+e_j),w(l)) = 1\ \text{whenever}\ l,l+e_j\in [m,n] \}.$$
Put $W_m=W_{[0,m]}$ if $m \ge 0$.
Say that an element $w \in W_m$ has \emph{ shape} $m$, and write $\s(w) = m$.
Thus $W_m$ is the set of words of shape $m$, and we identify $A$ with $W_0$ in the
natural way.
Define the initial and final maps $o: W_m \to A$ and $t: W_m \to A$ by
$o(w) = w(0)$ and $t(w) = w(m)$.
Fix a nonempty finite or countable set $D$ (whose elements are ``decorations''),
and a map $\delta : D \to A$.
Let $\overline W_m = \{ (d,w) \in D \times W_m ;\ o(w) = \delta (d) \}$, the set of
``decorated words'' of shape $m$, and identify
$D$ with $\overline W_0$ via the map $d \mapsto (d,\delta(d))$.
Let $W = \bigcup_m W_m$ and  $\overline W = \bigcup_m \overline W_m$, the sets of all
words and all decorated words respectively.
Define $o: \overline W_m \to D$ and $t: \overline W_m \to A$ by
$o(d,w) = d$ and $t(d,w) = t(w)$. Likewise extend the definition of shape to $\overline W$ by setting
$\s((d,w))=\s(w)$.

Given $j \le k\le l \le m$ and a function $w :[j,m] \to A$, define
$w \vert _{[k,l]} \in W_{l-k}$ by $w \vert _{[k,l]} = w'$ where
$w'(i) = w(i+k)$ for $0 \le i \le l-k$.
If $\overline w = (d,w) \in \overline W_m$, define
\begin{align*}
\overline w \vert_{[k,l]}& = w\vert_{[k,l]} \in  W_{l-k} \ \text{if}\ k \ne 0, \\
\text{and} \qquad
\ow \vert_{[0,l]}& = (d, w\vert_{[0,l]}) \in  \overline W_l.
\end{align*}
If $w\in W_l$ and $k\in \ZZ^r$, define $\tau_kw:[k,k+l] \to A$ by
$(\tau_kw)(k+j)=w(j)$. If $w \in W_l$ where $l \ge 0$ and if $p \ne 0$, say that
$w$ is {\em $p$-periodic} if its $p$-translate,
$\tau_pw$, satisfies $\tau_pw\vert_{[0,l]\cap[p,p+l]} = w\vert_{[0,l]\cap[p,p+l]}$.

Assume that the matrices $M_i$ have been chosen so that the following conditions hold.

\begin{description}
\item[(H0)] Each $M_i$ is a nonzero \zomatrix.

\item[(H1)] Let $u\in W_m$ and $v \in W_n$. If $t(u) =o(v)$ then there exists a unique
$w\in W_{m+n}$ such that
$$w\vert_{[0,m]}=u \qquad \text{and} \qquad w\vert_{[m,m+n]}=v.$$

\item[(H2)] Consider the directed graph which has a vertex for each $a \in A$
and a directed edge from $a$ to $b$ for each $i$ such that $M_i(b,a) =1$.
This graph is irreducible.

\item[(H3)] Let $p\in \ZZ^r$, $p \ne 0$. There exists some  $w \in W$ which is not $p$-periodic.
\end{description}

\begin{definition}\label{word-prod} In the situation of (H1) we write $w=uv$ and say
that the product $uv$ exists. This product is clearly associative.
\end{definition}

The $C^*$-algebra $\cA$ is defined as the universal $C^*$-algebra
generated by a family of partial isometries
$\{s_{u,v};\ u,v \in \overline W \ \text{and} \ t(u) = t(v) \}$
satisfying the relations
\begin{subequations}\label{rel1*}
\begin{eqnarray}
{s_{u,v}}^* &=& s_{v,u} \label{rel1a*}\\
s_{u,v}s_{v,w}&=&s_{u,w} \label{rel1b*}\\
s_{u,v}&=&\displaystyle\sum_
{\substack{w\in W;\s(w)=e_j,\\
o(w)=t(u)=t(v)}}
s_{uw,vw} ,\ \text{for} \ 1 \le j \le r
\label{rel1c*}\\
s_{u,u}s_{v,v}&=&0 ,\ \text{for} \ u,v \in \overline W_0, u \ne v. \label{rel1d*}
\end{eqnarray}
\end{subequations}

\section{Products of higher rank words}

Condition~(H1) is fundamental to all that follows.
How then, does one verify (H1)?
Given {\zomatrices} $M_i$, $1\le i\le r$, the following three simple conditions
will be seen to be sufficient.

\begin{description}
\item[(H1a)] $M_iM_j=M_jM_i$.

\item[(H1b)] For $i<j$, $M_iM_j$ is a {\zomatrix}.

\item[(H1c)] For $i<j<k$, $M_iM_jM_k$ is a {\zomatrix}.
\end{description}

Indeed, the first two conditions are also necessary.

\begin{lemma}
Fix {\zomatrices} $M_i$, $1\le i\le r$. Then {\rm (H1)} implies {\rm (H1a)} and {\rm (H1b)}.
\end{lemma}

\begin{proof}
Suppose $(M_iM_j)(b,a)>0$. Then there exists $c\in A$ so that $M_j(c,a)=1=M_i(b,c)$.
Let $u\in W_{e_j}$ and $v\in W_{e_i}$ be given by
$$u(0)=a\qquad u(e_j)=c\qquad \qquad \qquad v(0)=c\qquad v(e_i)=b.\qquad $$
According to (H1) there is a unique $w\in W_{e_i+e_j}$ with
$w(0)=a, w(e_j)=c, w(e_i+e_j)=b$. There must then be a unique $d\in A$ which can be used
for the missing value of $w$, $w(e_i)$. That is, there must be a unique $d\in A$ satisfying
$M_i(d,a)=1=M_j(b,d)$. Hence $(M_jM_i)(b,a)=1$.

We have seen that if $(M_iM_j)(b,a)>0$, then $(M_jM_i)(b,a)=1$. Likewise, if
$(M_jM_i)(b,a)>0$, then $(M_iM_j)(b,a)=1$. It follows that $M_iM_j$ and $M_jM_i$
are equal and have entries in $\{0,1\}$.
\end{proof}

\begin{lemma}\label{1}
Fix {\zomatrices} $M_i$ satisfying {\rm (H1a),(H1b)}, and {\rm (H1c)}.
Let $1 \le j \le r$. Let $w \in W_m$ and choose $a \in A$ so that $M_j(a,t(w)) =1$.
Then there exists a unique word $v \in W_{m+e_j}$ such that
$v\vert_{[0,m]} =w$ and $t(v) = a$.
\end{lemma}

\begin{proof}
In the case $r=2$ this follows from conditions (H1a) and (H1b) alone. The situation is
illustrated in Figure~\ref{extension}, for $j=2$. The assertion is
that there is a unique word $v$ defined on the outer rectangle $[0,m+e_2]$
with final letter  $v(m+e_2)=a$.
The hypothesis is that
there is a transition from $w(m)$ to $a$, in the sense that $M_2(a,w(m)) =1$.
Define $v(m+e_2)=a$. For notational convenience, let $n=m-e_1$.
 We have $M_1(w(m),w(n)) =1$, and the product matrix
$M_2M_1$ defines a transition $w(n) \to w(m) \to a$. The conditions (H1a) and (H1b) assert that
the product $M_1M_2$ defines a unique transition $w(n) \to b \to a$, for
some $b \in A$. Define $v(n+e_2)=b$. Continue the process inductively until
$v$ is defined uniquely on the whole of $[0,m+e_2]$.
This completes the proof if $r=2$.
\refstepcounter{picture}
\begin{figure}[htbp]
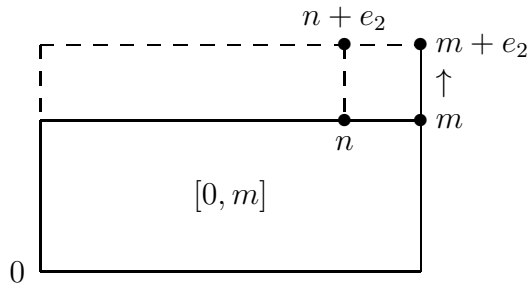
\label{extension}
\hfil
\centerline{
\beginpicture
\setcoordinatesystem units <1cm, 1cm>
\setplotarea  x from -6 to 6,  y from -1 to 2.2
\putrule from -3 -1 to 2 -1
\putrule from -3 1  to 2  1
\putrule from  -3 -1  to -3 1
\putrule from  2 -1  to 2 2
\setdashes
\putrule from  -3 2 to  2 2
\putrule from  -3 2 to  -3 1
\putrule from  1 2 to  1 1
\put {$[0,m]$}      at   -0.5  0
\put {$m$} [l]     at   2.2 1
\put {$0$} [r]     at   -3.2 -1
\put {$m+e_2$}[l]      at   2.2 2
\put {$n$} [t]     at   1  0.8
\put {$n+e_2$} [b]     at   1  2.2
\put {$\bullet$} at 2 1
\put {$\bullet$} at 2 2
\put {$\bullet$} at 1 1
\put {$\bullet$} at 1 2
\put{$\uparrow$}[l] at 2.2 1.5
\endpicture
}
\hfil
\caption{The case $r=2$.}
\end{figure}

Now consider the case $r=3$. Proceeding by induction
as in the case $r=2$, the extension problem reduces
to that for a single cube. Consider therefore
without loss of generality the unit cube based at $0$
with $m=e_1+e_2$ and $j=3$, as illustrated in Figure~\ref{extension3}.
Then $w$ is defined on the base of the cube $[0,m]$ and it is
required to extend $w$ to a function $v$ on the whole cube
taking the value $a$ at $m+e_3$, under the assumption
that there is a valid transition from $w(m)$ to $a$.
Now use the case $r=2$ on successive faces of the cube.
Working on the right hand face there is a unique possible value $b$
for $v(e_1+e_3)$. Then, using this value for $v(e_1+e_3)$ on the near face
we obtain the value $c$ for $v(e_3)$.
Similarly, working respectively on the back and left faces
we obtain a value $v(e_2+e_3)=d$ and a second value,  $c'$, for $v(e_3)$.

Now suppose that $c \ne c'$. Working on the top face and
using the values $a$, $d$, and $c'$, we obtain another value, $b'$ for $v(e_1+e_3)$.
There are two possible
transitions along the directed path $0 \to e_3 \to e_1+e_3 \to m+e_3$,
namely
$$w(0) \to c' \to b' \to a \quad \text{and} \quad w(0) \to c \to b \to a.$$
This contradicts the assumption that $(M_2M_1M_3)(a,w(0)) \in \{0,1\}$.

\refstepcounter{picture}
\begin{figure}[htbp]
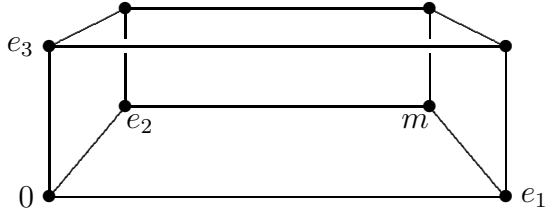
\label{extension3}
\hfil
\centerline{
\beginpicture
\setcoordinatesystem units <1cm, 1cm>
\setplotarea  x from -6 to 6,  y from -2 to 1
\putrule from -3 0 to 3 0
\putrule from -3 0 to -3 -2
\putrule from  3 0  to 3 -2
\putrule from -3 -2 to 3 -2
\putrule from -2 0.5 to  2 0.5
\setlinear \plot  -2 0.5  -3 0   /
\setlinear \plot  2 0.5  3 0  /
\putrule from -2 0.5 to -2 0.1
\putrule from -2 -0.1 to -2 -0.8
\putrule from 2 0.5 to 2 0.1
\putrule from 2 -0.1 to 2 -0.8
\putrule from -2 -0.8 to 2 -0.8
\setlinear \plot -2 -0.8  -3 -2 /
\setlinear \plot 2 -0.8  3 -2 /
\put {$\bullet$} at -3 0
\put {$e_3$} [r] at -3.2 0
\put {$\bullet$} at 3  0
\put {$\bullet$} at -3 -2
\put {$0$} [r] at -3.2 -2
\put {$\bullet$} at 3  -2
\put {$e_1$} [l] at 3.2  -2
\put {$\bullet$} at -2 0.5
\put {$\bullet$} at 2  0.5
\put {$\bullet$} at -2  -0.8
\put {$e_2$} [l,t] at -1.99  -0.9
\put {$\bullet$} at  2  -0.8
\put {$m$} [t,r]
at  1.99  -0.9
\endpicture
}
\hfil
\caption{The case $r=3$.}
\end{figure}

The proof for general $r$ now follows by induction. The uniqueness of the extension
follows from the two dimensional considerations embodied in Figure~\ref{extension}.
All the compatibility conditions required for existence follow from the three dimensional
considerations of Figure~\ref{extension3}.
\end{proof}

The next result follows by induction from Lemma~\ref{1}.

\begin{lemma}\label{2}
Fix {\zomatrices} $M_i$, $1\le i\le r$, satisfying {\rm (H1a), (H1b)}, and {\rm (H1c)}.
Let $1\le j_1,\dots ,j_p\le r$, let $a_0,\dots,a_p\in A$ and suppose that $M_{j_i}(a_i,a_{i-1})=1$
for $1\le i\le p$. Then there exists a unique word $w\in W$ with
$\sigma(w)=e_{j_1}+\dots+e_{j_p}$, such that
$w(0)=a_0$ and $w(e_{j_1}+\dots +e_{j_i})=a_i$ for $1\le i\le p$.
\end{lemma}
\noindent As a consequence we have

\begin{lemma}\label{3}
Fix {\zomatrices} $M_i$, $1\le i\le r$.
If  {\rm (H1a), (H1b)}, and {\rm (H1c)} hold, then {\rm (H1)} holds.
\end{lemma}

\begin{proof}
Let $u\in W_m$ and $v\in W_n$. Suppose  $t(u)=o(v)$. Choose $j_1,\dots,j_p$ as in Lemma~\ref{2}
so that
$$e_{j_1}+\dots+e_{j_q}=m \qquad  e_{j_{q+1}}+\dots+e_{j_p}=n$$
for some $q$, $0\le q\le p$. Choose $a_i$, $0\le i\le q$ so as to force the $w$ of Lemma~\ref{2} to
satisfy $w|_{[0,m]}=u$. Thus $a_q=t(u)=o(v)$. Choose $a_i$, $q\le i\le p$ so as to force $w$ to
satisfy $w|_{[m,m+n]}=v$. The existence and uniqueness in (H1) follow from the existence and uniqueness
 in Lemma~\ref{2}.
\end{proof}

\begin{corollary}\label{4}
If $\overline u=(d,u) \in \overline W_m$ and $v \in W_n$
with $t(\ou) = o(v)$,
then there exists a unique $\ow \in \overline W_{m+n}$ such that
$$\ow\vert_{[0,m]}=\ou \ \text{and} \ \ow\vert_{[m,m+n]}=v.$$
In these circumstances we write $\ow =\ou v$, and say
that the product $\ou v$ exists.
\end{corollary}

\begin{proof}
This is immediate, with $\ou v =(d,uv)$.
\end{proof}

For the next two lemmas, and for the rest of the paper, suppose that matrices\ $M_i$ have
been chosen so that (H0)--(H2) hold.

\begin{lemma}\label{ends} Let $a,b \in A$ and $n \in \ZZ^r_+$. There exists $w \in W$ with
$\s (w) \ge n$ such that $o(w)=a$ and $t(w)=b$.
\end{lemma}
\begin{proof}
By condition~(H0), the matrix $M_j$ is nonzero, so there exists at least one word of shape $e_j$.
Using this, choose words $w_1,\dots,w_q$ so that $\s(w_1)+\dots +\s(w_q)\ge n$.
Using conditions~(H1) and (H2),
one can always find a word with a given origin and terminus.
So choose $s_0 \in W$ with $o(s_0)=a$,  $t(s_0)=o(w_1)$, choose $s_k\in W$ with $o(s_k)=t(w_k)$, $t(s_k)=o(w_{k+1})$
for $1\le k \le q-1$, and choose $s_q\in W$ with $o(s_q)=t(w_q)$, $t(s_q)=b$.
Let $w=s_0w_1s_1w_2\dots w_{q-1}s_{q-1}w_qs_q$.
\end{proof}

\begin{lemma}\label{f1}
Given $\ou \in \oW$ and $b \in A$, there exists $v \in  W$ such that $\ou v$ exists,
$\s(v)\ne0$ and
\begin{equation*}t(v)=t(\ou v)=b.
\end{equation*}

\end{lemma}
\begin{proof}
This follows immediately Lemma~\ref{ends}.
\end{proof}

\section{nonperiodicity}\label{nonperiodicitysection}

Assume that $M_i$,\ $1\le i\le r$, have been chosen and that (H0)--(H2)
hold.  In the large class of examples associated to affine buildings
it is fairly easy to verify the nonperiodicity condition, (H3).
However, in general it is hard to see how one can start with the
matrices $M_i$ and check~(H3).  In this section we present a condition
which implies~(H3), and show how it can in principle be checked.  This
material is not used in the remainder of the paper.

\begin{description}
\item[(H3*)] Fix $j$, $1\le j\le r$. Let $m\in \ZZ_+^r$ with $m_j=0$. Let $w\in W_m$.
Then there exist $u, u' \in W_{m+e_j}$ such that
$u|_{[0,m]}=u'|_{[0,m]}=w$ but $u(e_j)\ne u'(e_j)$.
\end{description}

For $l,m \in \ZZ^r$, define
$$l \wedge m=(l_1 \wedge m_1, \dots, l_r \wedge m_r),$$
$$l \vee m=(l_1 \vee m_1, \dots, l_r \vee m_r),$$
$$|l|=l \vee (-l).$$

\noindent If $w \in W_l$ where $l \ge 0$ and if $p \ne 0$,
 recall that $w$ is {\em $p$-periodic} if its $p$-translate,
$\tau_pw$, satisfies $\tau_pw\vert_{[0,l]\cap[p,p+l]} = w\vert_{[0,l]\cap[p,p+l]}$.

\refstepcounter{picture}
\begin{figure}[htbp]
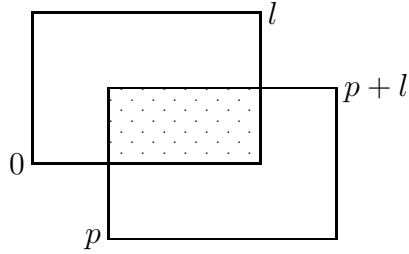
\label{intersection}
\hfil
\centerline{
\beginpicture
\setcoordinatesystem units <1cm, 1cm>
\setplotarea  x from -6 to 6,  y from -1 to 2
\setshadegrid span <4pt>
\putrectangle corners at -3 2 and 0 0
\putrectangle corners at -2 1 and 1 -1
\shaderectangleson
\putrectangle corners at -2 1 and 0 0
\shaderectanglesoff
\put{$0$}[r] at -3.1 0
\put{$l$}[l] at 0.1 2
\put{$p$}[r] at -2.1 -1
\put{$p+l$}[l] at 1.1 1
\endpicture
}
\hfil
\caption{The region $[0,l]\cap[p,p+l]$.}
\end{figure}

\begin{lemma}\label{a1} Conditions {\rm (H0)--(H2)} and {\rm (H3*)}
imply condition {\rm (H3)}.
\end{lemma}

\begin{proof}
Observe that $p$-periodicity is invariant under replacement of
$p$ by $-p$. We may therefore assume that $p$ has at least one positive
component which we may take to be $p_1$.
Let $p_{+}=p\vee 0$ and $p_{-}=(-p)\vee 0$.
We will construct $w\in W_{|p|}$. In this case $w$ is defined on $[0,|p|]$,
$\tau_pw$ is defined on $[p,|p|+p]$, and $\tau_pw$ and $w$ are both defined on
$[p\vee 0,|p|\wedge (|p|+p)]=[p_+,p_+]$, a single point.
The word $w$ is $p$-periodic if and only if
$w(p_+)=(\tau_pw)(p_+)=w(p_+-p)=w(p_-)$. The situation is illustrated in Figure~\ref{proofa1}.
Choose any $v$  in $W_{|p|-e_1}$
and let $u=v\vert_{[p_+-e_1,|p|-e_1]} \in W_{p_-}$.
By condition~(H3*), there exist two different words $x,x' \in W_{p_-+e_1}$ such that
$x\vert_{[0,p_-]}=x'\vert_{[0,p_-]}$ but $x(e_1)\ne x'(e_1)$.
At least one of $x(e_1)$ and $x'(e_1)$ differs from $v(p_-)$; we may assume that
$x(e_1)\ne v(p_-)$. Let $w$ be defined by $w\vert_{[0,|p|-e_1]}=u$, $w\vert_{[p_+-e_1,|p|]}=x$.
Then $w(p_+)=x(e_1)\ne v(p_-)=w(p_-)$, so $w$ is not $p$-periodic.
\end{proof}

\refstepcounter{picture}
\begin{figure}[htbp]
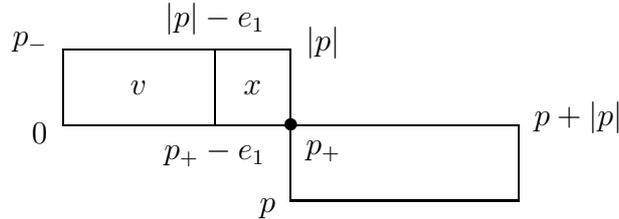
\label{proofa1}
\hfil
\centerline{
\beginpicture
\setcoordinatesystem units <1cm, 1cm>
\setplotarea  x from -3 to 3,  y from -1.1 to 1.1
\setshadegrid span <4pt>
\putrectangle corners at -3 1 and 0 0
\putrectangle corners at 0 0 and 3 -1
\putrectangle corners at -1 1 and 0 0
\put{$\bullet$} at 0 0
\put{$p_-$}[r] at -3.2 1.1
\put{$0$}[r] at -3.2 -0.1
\put{$|p|$}[l] at 0.2 1.1
\put{$p_+$}[l,t] at 0.2 -0.2
\put{$p$}[r] at -0.2 -1.1
\put{$p+|p|$}[l] at 3.2 0.1
\put{$|p|-e_1$}[b] at -1 1.2
\put{$p_+-e_1$}[t] at -1 -0.2
\put{$v$} at     -2 0.5
\put{$x$} at     -0.5 0.5
\endpicture
}
\hfil
\caption{Proof of Lemma~\ref{a1} in the case $r=2$.}
\end{figure}

Now we discuss the checkability of (H3*). Fix $j$, $1\le j\le r$.
For $w\in W_m$ with $\s(w)_j=0$ let
$$A(j,w)=\{u(e_j) ; u\in W_{m+e_j}\quad \text{and}\quad u|_{[0,m]}=w\}.$$
\noindent Given $w$, one can calculate $A(j,w)$ by considering, one at a time,
the possible values of $u(m+e_j)$, and working back to find the possible values of
$u(e_j)$ as in the proof of Lemma~\ref{1}. The assertion of (H3*) is that $\# A(j,w)\ge 2$
for any $w$ with $\s(w)_j=0$.

Let $v,w \in W$ with $\s(v)=e_k$, $k\ne j$, $\s(w)_j=0$, and suppose that $vw$ is defined.
Then $$A(j,vw)=\{a\in A ;\text{$M_j(a,o(v))=1$ and $M_k(b,a)=1$ for some $b\in A(j,w)$}\}.$$
\noindent Thus one can calculate $A(j,vw)$ from the knowledge of $v$ and $A(j,w)$.

To check (H3*) for the fixed value of $j$, one proceeds to construct,
for each $c\in A$ a complete list of possibilities for $A(j,w)$ with
$o(w)=c$. The first step in the algorithm is to insert in the lists
all $A(j,w)$ for $w\in W_0$. Then proceeding cyclically through all
words $v$ with $\s(v)=e_k$, for all $k\ne j$, the algorithm adds to
the lists all possible values of $A(j,vw)$ corresponding to values
$A(j,w)$ already on the lists.  The algorithm terminates when a
complete cycle through the words $v$ generates no new possible values
for $A(j,w)$.  The algorithm works because any $w\in W$ can be written
$w=v_1v_2\dots v_q$, with $\s(v_i) = e_{k_i}$.

Is this algorithm practical for hand computation? for electronic
computation? The authors have done no experiments, but they suspect
that the lists of subsets of~$A$ will get out of hand rapidly as the
cardinality of~$A$ increases.  The situation is not entirely
satisfactory.

\section{The $C^*$-algebra}

Assume conditions (H0)-(H3) hold.
Define an abstract untopologized algebra $\cA_0$ over $\CC$ which depends on
$A$, $(M_j)_{j=1}^r$, $D$, and $\delta$.
The generators of $\cA_0$ are
$\{s_{u,v}^0; u,v \in \overline W \ \text{and} \ t(u) = t(v) \}$.
The relations defining $\cA_0$ are
\begin{equation*}\label{rel0}
\begin{split}
s_{u,v}^0s_{v,w}^0&=s_{u,w}^0\\
s_{u,v}^0&=\displaystyle\sum_
{\substack{w\in W;\s(w)=e_j,\\
o(w)=t(u)=t(v)}}
s_{uw,vw}^0 ,\ \text{for} \ 1 \le j \le r \\
s_{u,u}^0s_{v,v}^0&=0 \ \text{for} \ u,v \in \overline W_0,  u \ne v.
\end{split}
\end{equation*}

It is trivial to verify that $\cA_0$ has an antilinear antiautomorphism defined
on the generators by

\begin{equation*}
{s_{u,v}^0}^* = s_{v,u}^0.
\end{equation*}
This makes $\cA_0$ a $*$-algebra.
Let $\cA$ be the corresponding enveloping $C^*$-algebra (c.f. \cite[p.256]{ck})
and let $s_{u,v}$ be the image of $s_{u,v}^0$ in $\cA$.
The generators of $\cA$ are therefore
\begin{equation*}
\{s_{u,v};\ u,v \in \overline W \ \text{and} \ t(u) = t(v) \}
\end{equation*}
and the defining relations are

\begin{subequations}\label{rel1}
\begin{eqnarray}
{s_{u,v}}^* &=& s_{v,u} \label{rel1a}\\
s_{u,v}s_{v,w}&=&s_{u,w} \label{rel1b}\\
s_{u,v}&=&\displaystyle\sum_
{\substack{w\in W;\s(w)=e_j,\\
o(w)=t(u)=t(v)}}
s_{uw,vw} ,\ \text{for} \ 1 \le j \le r
\label{rel1c}\\
s_{u,u}s_{v,v}&=&0 \ \text{for} \ u,v \in \overline W_0, u \ne v .\label{rel1d}
\end{eqnarray}
\end{subequations}

\begin{remark}\label{equiv} Suppose that $u,v \in \overline W$ and
$t(u)=t(v)$.  Then  $s_{u,v}$ is a partial isometry with initial projection
${s_{u,v}}^*s_{u,v}=s_{v,v}$ and final projection $s_{u,v}{s_{u,v}}^*=s_{u,u}$.
\end{remark}

\begin{lemma}\label{f2} Fix $m \in \ZZ^r_+$ and let
$u,v \in \overline W$ with $t(u) = t(v)$.
Then
\begin{equation*}s_{u,v}= \displaystyle\sum_
{\substack{
w\in W;\s(w)=m \\
o(w)=t(u)=t(v)
}} s_{uw,vw}.
\end{equation*}
\end{lemma}

\begin{proof}
Using (H1), this follows by induction from (\ref{rel1c}).
\end{proof}

\begin{lemma}\label{5} $s_{u,u}s_{v,v}=0$ if $\s(u)=\s(v)$ and $u \ne v$.
\end{lemma}

\begin{proof}
The case $\s(u)=\s(v)=0$ is exactly the relation (\ref{rel1d}).
Assume that the assertion is true whenever $\s(u)=\s(v)=m$.
Let  $\s(u')=\s(v')=m+e_j$ and let $u=u'\vert_{[0,m]}$, $v=v'\vert_{[0,m]}$.
By relation (\ref{rel1c}), we have that
$s_{u,u}=\sum_w s_{uw,uw}$ where the sum is over $w \in W_{e_j}$ such that
$o(w)=t(u)$. Since $s_{u',u'}$ is one of the terms of the
preceding sum we have $s_{u,u} \ge s_{u',u'}$. Similarly $s_{v,v} \ge s_{v',v'}$.
If $u \ne v$, this proves that $s_{u',u'}s_{v',v'}=0$, since by induction $s_{u,u}s_{v,v}=0$.
On the other hand, if $u=v$ then  $s_{u',u'}$, $s_{v',v'}$ are distinct terms
in the sum $\sum_a s_{uw,uw}$ and are therefore orthogonal.
\end{proof}

\begin{remark}\label{finitedecorate} If\ $\oW_0=D$ is finite then it follows from {\rm (\ref{rel1d})} that
$\sum_{u\in \oW_0}s_{u,u}$ is an idempotent.
From Lemma~\ref{f2} it follows that for any $m$,
$\sum_{u\in \oW_m}s_{u,u}=\sum_{u\in \oW_0}s_{u,u}$.
Hence from {\rm (\ref{rel1b})} and Lemma \ref{5}
it follows that $\sum_{u\in \oW_0}s_{u,u}$ is an identity for $\cA$.
\end{remark}

The next lemma is an immediate consequence of the definition of an
enveloping $C^*$-algebra.

\begin{lemma}\label{6}
Let ${\cH}$ be a Hilbert space and for each $u,v \in \overline W$ with
$t(u) = t(v)$ let $S_{u,v} \in \cB(\cH)$. If the $S_{u,v}$
satisfy the relations (\ref{rel1}), then there is a unique *-homomorphism
$\phi : {\cA} \to \cB(\cH)$ such that
$\phi(s_{u,v})=S_{u,v}$. \qed
\end{lemma}

\begin{lemma}\label{7} Any product $s_{u_1,v_1}s_{u_2,v_2}$ can be written as a finite
sum of the generators $s_{u,v}$.
\end{lemma}

\begin{proof}
Choose $m \in \ZZ^r$ with $m \ge \s(v_1)$ and $m \ge \s(u_2)$.
Use Lemma~\ref{f2} to write $s_{u_1,v_1}$ as a sum of terms $s_{u_3,v_3}$ with $\s(v_3)=m$.
 Likewise, write
$s_{u_2,v_2}$ as a sum of terms $s_{u_4,v_4}$ with $\s(u_4)=m$.
Now in the product $s_{u_1,v_1}s_{u_2,v_2}$ each term has the form
$$s_{u_3,v_3}s_{u_4,v_4}=\begin{cases}
s_{u_3,v_4}& \text{if $v_3=u_4$},\\
s_{u_3,v_3}s_{v_3,v_3}s_{u_4,u_4}s_{u_4,v_4}=0& \text{if $v_3 \ne u_4$}
\end{cases}$$
by Lemma~\ref{5}. The result follows immediately.
\end{proof}

\begin{corollary}\label{8} The $C^*$-algebra ${\cA}$ is the closed linear span
of the set
\begin{equation*}
\{s_{u,v}; u,v \in \overline W \ \text{and} \ t(u) = t(v) \}.
\end{equation*}
\qed
\end{corollary}

\begin{lemma}\label{nonzero} The algebra $\cA$ is nonzero.
\end{lemma}

\begin{proof}
We must construct a nonzero *-homomorphism from $\cA$ into
$\cB(\cH)$ for some Hilbert space $\cH$.
Consider the set of infinite words
$$W_{\infty} = \{ w: \ZZ^r_+ \to A ; M_j(w(l+e_j),w(l)) = 1\ \text{whenever}\ l \ge 0 \},$$
and define the product
$uv$ for $u \in \overline W$ and $v \in W_{\infty}$ exactly as in Definition~\ref{word-prod}.
Let $\cH = l^2(W_{\infty})$ and define
$$(\phi(s_{u,v}))(\delta_w)=\begin{cases}
\delta_{uw_1} & \text{if $w=vw_1$ for some $w_1 \in W_{\infty}$,}\\
0& \text{otherwise}.
\end{cases}$$
It is easy to check that the operators $\{\phi(s_{u,v});\ u,v \in \overline W \ \text{and} \ t(u) = t(v) \}$
satisfy the relations~\ref{rel1} and it
follows from Lemma~\ref{6} that $\phi$ extends to a *-homomorphism of $\cA$.
\end{proof}

\begin{remark}
The Hilbert space $\cH = l^2(W_{\infty})$ is not separable. However, since the algebra
$\cA$ is countably generated, there exist nonzero separable, $\cA$-stable subspaces of $\cH$,
 and in particular,
there exist nontrivial representations of $\cA$ on  separable Hilbert space.
\end{remark}

\begin{lemma}\label{nonzero+} If $\phi$ is a nontrivial representation of $\cA$, and if $u \in \oW$
then $\phi(s_{u,u})\ne 0$. In particular $s_{u,u} \ne 0$.
\end{lemma}

\begin{proof}
Suppose $\phi(s_{u,u}) = 0$. Choose any $v\in \oW$. Use Lemma~\ref{ends} to
find $w \in W$ such that $o(w)=t(u)$ and $t(w)=t(v)$. By Lemma~\ref{f2}, we have
$0 =\phi(s_{u,u}) = \sum \phi(s_{uw',uw'})$, the sum being taken over all $w' \in W$ such that
$\s(w')=\s(w)$ and $o(w')=t(u)$. Thus $\phi(s_{uw,uw})=0$. By  Remark~\ref{equiv}
it follows that that $\phi(s_{uw,v})=0$, that $\phi(s_{v,v})=0$, and finally that
$\phi(s_{v,v'}) = 0$ whenever $t(v)=t(v')$.
Hence $\phi$ is trivial.
\end{proof}

\begin{remark} When $r=1$, the algebra $\cA$ is a simple Cuntz-Krieger algebra.
More precisely, if we write $M=M_1^t$, then the Cuntz-Krieger algebra $\cO_M$ is generated by
a set of partial isometries $\{S_a ; a\in A\}$ satisfying the relations
$S_a^*S_a=\sum_bM(a,b)S_bS_b^*$. If $u\in W$, let
$S_u=S_{u(0)}S_{u(1)}\dots S_{t(u)}$ and if $v\in W$ with $t(u)=t(v)$, define
$S_{u,v}=S_uS_v^*$ (c.f. \cite[Lemma 2.2]{ck}).
The map $s_{u,v}\mapsto S_{u,v}$ establishes an isomorphism of $\cA$ with $\cO_M$.
Tensor products of ordinary Cuntz-Krieger algebras can be identified as higher rank Cuntz-Krieger
algebras $\cA$. If $\cA_1,\cA_2$ are simple rank one Cuntz-Krieger algebras, with
corresponding matrices $M_1, M_2$ and alphabets $A_1, A_2$
then $\cA_1 \otimes \cA_2$ is the algebra $\cA$ arising from
the pair of matrices $M_1\otimes I, I\otimes M_2$ and the alphabet $A_1 \times A_2$.
More interesting examples arise from group actions on affine buildings.
The details for some $r=2$ algebras arising in this way are given in Section~\ref{boundary-algebra}.
\end{remark}

\section{The AF subalgebra}\label{The AF subalgebra}

If $m \in \ZZ^r_+$, let  $\cF_m$ denote the subalgebra of $\cA$ generated by the
elements $s_{u,v}$ for $u,v \in \overline W_m$.

\begin{lemma}\label{fm} There exists an isomorphism
$\cF_m\cong \bigoplus_{a\in A} \cK(l^2(\{w \in \overline W; \s(w)=m, t(w)=a\})).$
\end{lemma}
\begin{proof}
Let $u,v \in \oW_m$ with $t(u)=t(v)=a$. Consider the map
$E^a_{\delta_u,\delta_v}\mapsto s_{u,v}$, where $E^a_{\delta_u,\delta_v}$
denotes a standard matrix unit in $\cK(l^2(\{w \in \overline W; \s(w)=m, t(w)=a\}))$.
This extends to a map which is an isomorphism according to equations (\ref{rel1a}),
(\ref{rel1b}) and Lemma~\ref{5}.
\end{proof}
The relations (\ref{rel1c}) show that there is
a natural embedding of $\cF_m$ into $\cF_{m+e_j}$.
The $C^*$-algebras $\{\cF_m: m \in \ZZ^r_+\}$ form a directed
system of  $C^*$-algebras in the sense of \cite[p. 864]{kr}.
By \cite[Proposition 11.4.1]{kr} there is an essentially unique $C^*$-algebra
$\cF$  in which the union of the algebras $\cF_m$ is dense, namely
the direct limit of these algebras. We have the following
commuting diagram of inclusions.

\begin{equation*}
\begin{CD}
\cF_{m+e_k}   @>>>   \cF_{m+e_j+e_k}\\
@AAA                                  @AAA\\
\cF_m                  @>>>      \cF_{m+e_j}
\end{CD}
\end{equation*}

We may equally well regard $\cF$ as the closure of
$\bigcup_{j=1}^{\infty}\cF_{jp}$, where $p=(1,1,\dots,1)$. In
particular $\cF$ is an $AF$-algebra.

\begin{proposition}\label{F}
Let $\phi$ be a nonzero homomorphism from $\cA$ into some $C^*$-algebra.
Then the restriction of $\phi$ to $\cF$ is an isomorphism.
\end{proposition}
\begin{proof}
Suppose that $\phi$ is not an isomorphism on $\cF$.
Then $\phi$ is not an isomorphism on some $\cF_m$. Since
$\bigoplus_{a\in A} \cK(l^2(\{w \in \overline W; \s(w)=m, t(w)=a\})) \cong \cF_m$ (Lemma~\ref{fm}),
it follows from simplicity of the algebra of compact operators that
$\phi(s_{u,u})=0$ for some $u$, contradicting Lemma~\ref{nonzero+}.
\end{proof}

Define an action $\alpha$ of the $r$-torus $\TT^r$ on $\cA$ as follows.
If $\s(u)-\s(v) =m \in \ZZ^r$ and $t=(t_1,\dots,t_r) \in \TT^r$,
let $\alpha_t(s_{u,v}) = t^ms_{u,v}$, where $t^m=t_1^{m_1}t_2^{m_2}\dots t_r^{m_r}$.
The elements $\alpha_t(s_{u,v})$ satisfy the relations (\ref{rel1})
and generate the $C^*$-algebra $\cA$. By the universal property of $\cA$
it follows that $\alpha_t$ extends to an automorphism
of $\cA$. It is easy to see that that $t\mapsto \alpha_t$ is an action.
It is also clear that $\alpha_t$ fixes all elements $s_{u,v}$ with $\s(u)=\s(v)$
and so fixes $\cF$ pointwise. We now show that $\cF =\cA^{\alpha}$,
the fixed point subalgebra of $\cA$. Consider the linear map on $\cA$ defined by
\begin{equation}\label{expectation}
\pi(x)= \int_{\TT^r}\alpha_t(x)dt, \qquad \text{for} \ x \in \cA
\end{equation}
where $dt$ denotes normalized Haar measure on $\TT^r$.

\begin{lemma}\label{b} Let $\pi$, $\cF$ be as above.
\begin{enumerate}
\item The map $\pi$ is a faithful conditional expectation from $\cA$ onto $\cF$.
\item $\cF =\cA^{\alpha}$, the fixed point subalgebra of $\cA$.
\end{enumerate}
\end{lemma}

\begin{proof}
Since the action $\alpha$ is continuous, it
is easy to see that $\pi$ is a conditional expectation from $\cA$ onto $\cA^{\alpha}$,
and that it is faithful.
 Since $\alpha_t$ fixes $\cF$ pointwise,
$\cF \subset \cA^{\alpha}$. To complete the proof we show that the range
of $\pi$ is contained in (and hence equal to) $\cF$. By continuity of $\pi$ and
Corollary~\ref{8}, it is enough to show that $\pi(s_{u,v}) \in \cF$ for all
$u,v \in \overline W$. This is so because
\begin{equation}\label{pi(x)} \pi(s_{u,v})= s_{u,v}\int_{\TT^r}t^{\s(u)-\s(v)}dt=\begin{cases}
0& \text{if $\s(u)\ne \s(v)$},\\
s_{u,v}& \text{if $\s(u) = \s(v)$}.
\end{cases}
\end{equation}
\end{proof}

\section{Simplicity}\label{sectionsimplicity}

We show that the $C^*$-algebra $\cA$ is simple. Consequently any nontrivial $C^*$-algebra with
generators $S_{u,v}$ satisfying relations (\ref{rel1}) is isomorphic to $\cA$.
Several preliminary lemmas are necessary. Consequences of (H3), their theme is the existence
of words lacking certain periodicities.
Recall that for $m\in \ZZ^r$, $|m|=(|m_1|,\dots ,|m_r|)$.

\begin{lemma} \label{a2}
Let $m \in \ZZ^r$ with $m \ge 0$ and let $a \in A$.
There exists some $l \ge 0$ and some $w \in W_l$
satisfying
\begin{enumerate}
\item If $|p| \le m$ and $p \ne 0$ then
$\tau_pw\vert_{[0,l]\cap[p,p+l]} \ne w\vert_{[0,l]\cap[p,p+l]}$,
\item  $o(w)=a.$
\end{enumerate}
\end{lemma}

\begin{proof} Note that if $w=uw_pv$, and if $w_p$ is not $p$-periodic, then neither is $w$.
Apply (H3) to obtain for each nonzero $p$, $|p| \le m$, a word $w_p$
which is not $p$-periodic.
The final word $w$ is obtained by concatenating these words
in some order using ``spacers'' whose existence is guaranteed by Lemma~\ref{ends}.  The construction is
illustrated in Figure~\ref{concatenate}, where the spacer $s_{p,q}$ is chosen so that $o(s_{p,q})=t(w_p)$
and $t(s_{p,q})=o(w_q)$.
\end{proof}

\refstepcounter{picture}
\begin{figure}[htbp]
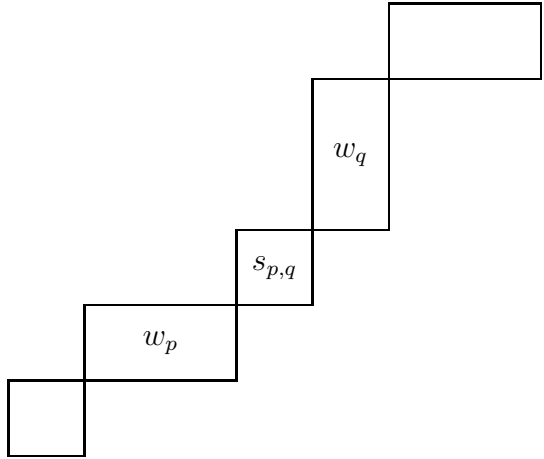
\label{concatenate}
\hfil
\centerline{
\beginpicture
\setcoordinatesystem units <1cm, 1cm>
\setplotarea  x from -6 to 6,  y from -3 to 3
\putrectangle corners at -3 -3 and -2 -2
\putrectangle corners at -2 -2 and 0 -1
\putrectangle corners at 0 -1 and 1 0
\putrectangle corners at 1 0 and 2 2
\putrectangle corners at 2 2 and 4 3
\put{$w_p$} at -1  -1.5
\put{$s_{p,q}$} at 0.5  -0.5
\put{$w_q$} at 1.5  1
\endpicture
}
\hfil
\caption{Part of the word $w$; a word $s_{p,q}$ is used to concatenate $w_p$ and $w_q$.}
\end{figure}

\begin{lemma}\label{extra}
One can find $u,u'\in W$ with
$\s(u)=\s(u')$, $o(u)=o(u')$, but $u\ne u'$.
\end{lemma}

\begin{proof}
Assume the contrary. Considering words of shape $e_1$, we see that for fixed $a$, no more than one
$b\in A$ satisfies $M_1(b,a)=1$. By Lemma~\ref{ends}, at least one $b\in A$ satisfies $M_1(b,a)=1$
and at least one $c\in A$ satisfies $M_1(c,a)=1$. Consequently
the directed graph associated to $M_1$ must be a union of closed cycles.
If $g$ is the g.c.d of the cycle lengths, then every $w\in W$ is $ge_1$-periodic, contradicting (H3).
\end{proof}

\begin{lemma} \label{a3}
Let $p \in \ZZ^r$. Let $w_1,w_2 \in  W_l$.
There exist $l' \ge l$ and $w_1',w_2' \in  W_{l'}$
such that
$$w_1'\vert_{[0,l]}=w_1, \qquad w_2'\vert_{[0,l]}=w_2,$$
and
$$\tau_pw_1'\vert_{[0,l']\cap[p,p+l']} \ne w_2'\vert_{[0,l']\cap[p,p+l']}.$$
\end{lemma}

\begin{proof}
Find two different words $u,v$ with $\s(u)=\s(v)$ and $o(u)=o(v)$.
Choose $s$ so that $p+\s(w_1)+\s(s) \ge 0$, $o(s)=t(w_1)$ and $t(s)=o(u)=o(v)$.
Choose $w_2''$ so that $w_2''\vert_{[0,l]}=w_2$ and
$\s(w_2'') \ge p+\s(w_1)+\s(s)+\s(u)$.
Consider $w_2''\vert_{[p+\s(w_1)+\s(s),p+\s(w_1)+\s(s)+\s(u)]}$.
If this is equal to $u$, let $w_1''=w_1sv$,
otherwise, let
$w_1''=w_1su$. (This is illustrated in Figure~\ref{aperiodiclemma}.)
Finally, let $l'=\s(w_1'')\vee \s(w_2'')$
and extend $w_1''$ and $w_2''$ to words $w_1'$ and $w_2'$ in
$ W_{l'}$.
\end{proof}

\refstepcounter{picture}
\begin{figure}[htbp]
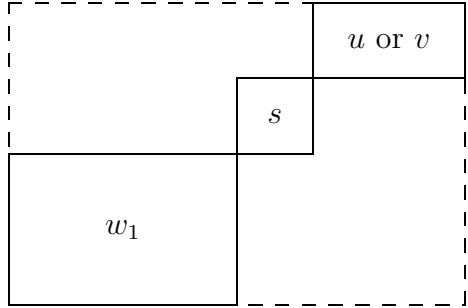
\label{aperiodiclemma}
\hfil
\beginpicture
\setcoordinatesystem units <1cm, 1cm>
\setplotarea  x from -6 to 6,  y from -2 to 2
\putrectangle corners at -3 -2 and 0 0
\putrectangle corners at 0 0 and 1 1
\putrectangle corners at 1 1 and 3 2
\setdashes
\putrectangle corners at -3 -2 and 3 2
\put{$w_1$} at -1.5  -1
\put{$s$} at 0.5  0.5
\put{$u$ or $v$} at 2  1.5
\endpicture
\hfil
\caption{The word  $w_1''$.}
\end{figure}

\begin{lemma} \label{bb}
Fix $m \in \ZZ^r$ with $m \ge 0$. For some $l \ge 0$
there exists a subset
$S =\{w_a;\ a \in A\}$ of $W_l$
satisfying the two properties below.
\begin{enumerate}
\item For each $a \in A$, $o(w_a)=a.$
\item Let $a,b \in A$. Let $p \ne 0$ be in $\ZZ^r$ with
$|p| \le m$.
Then $w_a\vert_{[0,l]\cap[p,p+l]} \ne \tau_pw_b\vert_{[0,l]\cap[p,p+l]}$.
\end{enumerate}
\end{lemma}

\begin{proof}
The elements of $S$ are chosen as follows. For each $a \in A$,
let $w_a \in W$ satisfy the conclusions of Lemma~\ref{a2}.
By extending the words $w_a$ as necessary we may suppose
that $w_a \in W_l$ for some $l \ge 0$.
If $a,b\in A$, $a\ne b$ and $p\in \ZZ^r$ with $|p|\le m$, we can apply
Lemma~\ref{a3} and extend $w_a$ and $w_b$ to $w_a'$ and $w_b'$ where $\tau_pw_a'$
and $w_b'$ do not agree on their common domain. Then one can extend all the other
$w_c$ to $w_c'$ of the same shape  as $w_a'$ and $w_b'$. Apply this procedure once for
each of the finitely many triples $(a,b,p)\in A\times A\times \ZZ^r$ with $|p|\le m$ and
$a\ne b$, and the proof is done.
\end{proof}

Fix $m \in \ZZ^r_+$.
Choose $l$ and $S$ satisfying the conditions of Lemma~\ref{bb}. Define
\begin{equation}\label{EQ}Q= \displaystyle\sum_
{\substack{
w \in \overline W;\s(w)=m+l \\
w\vert_{[m,m+l]} \in S
}} s_{w,w}.
\end{equation}

\begin{lemma}\label{bc}
If $l'\ge l$, then
\begin{equation}\label{XX}
Q= \displaystyle\sum_
{\substack{
w\in \overline W;\s(w)=m+l' \\
w\vert_{[m,m+l]} \in S
}} s_{w,w}.
\end{equation}
\end{lemma}

\begin{proof} By (H1) and Lemma~\ref{f2},
\begin{equation*}
\displaystyle\sum_
{\substack{
w\in \overline W;\s(w)=m+l' \\
w\vert_{[m,m+l]} \in S
}} s_{w,w}
=
\displaystyle\sum_
{\substack{
w_1\in \overline W, w_2\in W \\
\s(w_1)=m+l, \s(w_2)=l'-l  \\
w_1\vert_{[m,m+l]} \in S, t(w_1)=o(w_2)
}} s_{w_1w_2,w_1w_2}
=
Q.
\end{equation*}
\end{proof}

\begin{lemma}\label{bd}
Suppose $\s(u),\s(v) \le m$ and $t(u)=t(v)=a$.
If $\s(u) \ne \s(v)$, then $Qs_{u,v}Q = 0$.
\end{lemma}

\begin{proof} Using Lemma~\ref{f2}, write
\begin{equation*}s_{u,v}= \displaystyle\sum_
{\substack{
v_1; \s(v_1)=m+l \\
o(v_1)=a
}} s_{uv_1,vv_1}.
\end{equation*}
Apply Lemma~\ref{bc} with $l'=l+\s(v)$ so that
for each word $w$ in the sum (\ref{XX}) $\s(w)=m+l+\s(v)=\s(vv_1)$.
By Lemma~\ref{5}, $s_{uv_1,vv_1}s_{w,w}=0$ unless $vv_1=w$.
Consequently $s_{uv_1,vv_1}Q=0$ unless $vv_1\vert_{[m,m+l]} \in S$,
which is to say $v_1\vert_{[m-\s(v),m-\s(v)+l]}\in S$.
Similarly, $Qs_{uv_1,vv_1}=0$ unless
$v_1\vert_{[m-\s(u),m-\s(u)+l]}\in S$.
But if $w_1=v_1\vert_{[m-\s(v),m-\s(v)+l]}\in S$ and
$w_2=v_1\vert_{[m-\s(u),m-\s(u)+l]}\in S$,
then $w_1$ and $w_2$ would fail condition (2) of Lemma~\ref{bb},
with $p=\s(u)-\s(v)$.
\end{proof}

\begin{remark}\label{bd+}
If $x=\sum_ic_is_{u_i,v_i}$ is a finite linear combination of the generators of $\cA$
with $\s(u_i),\s(v_i)\le m$ then
$QxQ= \displaystyle\sum_{\s(u_i)=\s(v_i)}c_iQs_{u_i,v_i}Q=Q\pi(x)Q$, by Lemma~\ref{bd}
and equation {\rm(}\ref{pi(x)}{\rm)}.
\end{remark}

\begin{lemma}\label{be}
The map $x \mapsto QxQ$ is an isometric *-algebra map from $\cF_m$
into $\cA$.
\end{lemma}

\begin{proof}
If $\s(u)=\s(v)=m$ and $t(u)=t(v)=a$, then, with the notation of Lemma~\ref{bb},
we have $Qs_{u,v}Q= s_{uw_a,vw_a}$.
With the notation of Lemma~\ref{fm}, consider the map
$$\bigoplus_{a\in A} \cK(l^2(\{w \in \overline W; \s(w)=m, t(w)=a\})) \to \cA$$
given by
$$E_{\delta_u,\delta_v}^a \mapsto s_{uw_a,vw_a},$$
when $t(u)=t(v)=a$. This is easily checked to be an injective *-algebra map,
hence an isometry. The map in the statement of the Lemma is the composition of this isometry
with that of Lemma~\ref{fm}.
\end{proof}

\begin{theorem}\label{main1} The $C^*$-algebra $\cA$ is simple.
\end{theorem}
\begin{proof}
Let $\phi$ be a nonzero *-homomorphism from $\cA$ to some $C^*$-algebra.
It is enough to show that $\phi$  is an isometry.
Let $x=\sum_ic_is_{u_i,v_i}$ be a finite linear combination of the generators of $\cA$.
Choose $m \in \ZZ^r_+$ so that $\s(u_i),\s(v_i) \le m$ for all $i$.
Choose $l$ and $S$ as in Lemma~\ref{bb} and let $Q$ be as in equation (\ref{EQ}).
By Remark~\ref{bd+}, $QxQ=Q\pi(x)Q$.
Observe that  $\pi(x) \in \cF$ and so by Lemma~\ref{be},  $\|Q\pi(x)Q\| = \|\pi(x)\|$.
Moreover $Q\pi(x)Q \in \cF_{m+l}$, so by Proposition~\ref{F},
$\|\phi(Q\pi(x)Q)\| =\|Q\pi(x)Q\|$. Thus
$$\|\phi(x)\| \ge \|\phi(Q)\phi(x)\phi(Q)\| = \|\phi(QxQ)\| =
\|\phi(Q\pi(x)Q)\|=\|Q\pi(x)Q\|=\|\pi(x)\|.$$
The inequality extends by continuity to all $x \in \cA$.
It follows that $\phi$ is faithful since if $\phi(y)=0$ then
$0 = \|\phi(y^*y)\| \ge \|\pi(y^*y)\|$. Therefore $y=0$,
by Lemma~\ref{b}.
\end{proof}

\begin{corollary}\label{main1+}
Let ${\cH}$ be a Hilbert space and for each $u,v \in \overline W$ with
$t(u) = t(v)$ let $S_{u,v} \in \cB(\cH)$ be a nonzero partial isometry.
If the $S_{u,v}$ satisfy the relations (\ref{rel1}) and $\hat \cA$
denotes the $C^*$-algebra which they generate,
then there is a unique *-isomorphism
$\phi$ from $\cA$ onto $\hat \cA$  such that
$\phi(s_{u,v})=S_{u,v}$.
\end{corollary}

\begin{proof} This follows immediately from Lemma~\ref{6} and Theorem~\ref{main1}.
\end{proof}

\begin{proposition}\label{purely_infinite}
The $C^*$-algebra $\cA$ is purely infinite.
\end{proposition}

\begin{proof} Note first that by Lemma~\ref{f2}, $s_{uw,uw}$ is a
subprojection of $s_{u,u}$. By Lemma~\ref{f1}, this implies that $s_{u,u}$
is an infinite projection for any $u$. Any rank one projection in any $\cF_l$
is equivalent to $s_{u,u}$ for some $u\in W_l$, and hence is infinite.

We must show that for every nonzero $h \in \cA_+$, the $C^*$-algebra
$\overline{h\cA h}$ contains an infinite projection.
Since $\pi$ is faithful, we may assume $\|\pi(h^2)\|=1$.
Let $0 < \epsilon <1$.
Approximate $h$ by self-adjoint finite linear combinations of generators.
The square of this approximation gives an element  $y=\sum_ic_is_{u_i,v_i}$
with $y\ge 0$ and $\| y-h^2\|\le\epsilon$. Fix $m$ so that $\s(u_i),\s(v_i)\le m$ for all $i$ and
then construct $Q$ by Lemma~\ref {bb} and equation (\ref {EQ}).
We have $QyQ=Q\pi(y)Q\in \cF_l$ for some $l$, $QyQ\ge 0$ and
$\|QyQ\|= \|Q\pi(y)Q\|=\|\pi(y)\|\ge\|\pi(h^2)\|-\epsilon=1-\epsilon$.
Since $\cF_l$ is a direct sum of (finite or infinite dimensional) algebras of compact operators,
there exists a rank one positive operator $R_1\in \cF_l$ with $\|R_1\|\le(1-\epsilon)^{-1/2}$ so that
$R_1QyQR_1=P$ is a rank one projection in $\cF_l$. Hence $P$ is an infinite projection.

It follows that
$\|R_1Qh^2QR_1-P\|\le \|R_1^2\|\|Q\|^2\|y-h^2\|\le \epsilon /(1-\epsilon)$.
By functional calculus, one obtains $R_2\in \cA_+$ so that $R_2R_1Qh^2QR_1R_2$
is a projection and $\|R_2R_1Qh^2QR_1R_2-P\|\le 2\epsilon /(1-\epsilon)$.
For small $\epsilon$ one can then find an element $R_3$ in $\cA$ so that
$R_3R_2R_1Qh^2QR_1R_2R_3^*=P$.

Let $R=R_3R_2R_1Q$, so that $Rh^2R^*=P$. Consequently, $Rh$ is a partial isometry, whose initial projection
$hR^*Rh$ is a projection in $h\cA h$ and whose final projection is $P$. Moreover, if $V$ is a partial isometry in $\cA$
such that $V^*V=P$ and $VV^* < P$, then $(hR^*)V(Rh)$ is a partial isometry
in $h\cA h$ with initial projection $hR^*Rh$ and final projection strictly less than
$hR^*Rh$.
\end{proof}

We can now explain one of the reasons for introducing the set $D$ of decorations.
Recall that $D$ is a countable or finite set.
Denote by $\cA_D$ the algebra $\cA$ corresponding to a given $D$.
One special case of interest is when $D=A$  and $\delta$ is the
identity map. Then the algebra $\cA_A$ is a direct generalization of a Cuntz-Krieger
algebra \cite{ck}. There is an obvious notion of equivalence for decorations: two decorations
$\delta_1 : D_1 \to A$ and $\delta_2 : D_2 \to A$ are equivalent if there is a bijection
$\eta: D_1 \to D_2$ such that $\delta_1 = \delta_2 \eta$. Equivalent decorations give rise to
isomorphic algebras. Given any set $D$ of decorations we can obtain another
set of decorations $D \times \NN$, with the decorating map $\delta':D\times \NN
\to A$ defined by $\delta'((d,i)) = \delta(d)$.

\begin{lemma}\label{DN} There exists an isomorphism of $C^*$-algebras $\cA_{D \times \NN} \cong \cA_D\otimes\cK$.
\end{lemma}
\begin{proof}
If $u,v \in W$, the isomorphism is given by
$s_{((d,i),u),((d',j),v)} \mapsto s_{(d,u),(d',v)}\otimes E_{i,j}$,
where the $E_{i,j}$ are matrix units for $\cB(l^2(\NN))$. The fact that this
is an isomorphism follows from Corollary~\ref{main1+}.
\end{proof}

This procedure is useful, because it provides a routine method of passing from
$\cA$ to $\cA \otimes\cK$, a technique that is necessary to obtain the results of
\cite{ck}.

\begin{lemma}\label{27}
Let $l: D \to \ZZ^r_+$ be any map. Define $D'=\{(d,w)\in \oW;\ \s(w)=l(d)\}$
and define $\dd: D' \to A$ by $\dd(\ow)=t(\ow)$. Then $\cA_{D'}\cong\cA_D$.
\end{lemma}

\begin{proof}
Define $\phi: \cA_{D'}\to\cA_D$ by $\phi(s_{(\ow_1,u_1),(\ow_2,u_2)})=s_{\ow_1u_1,\ow_2u_2}$,
for $\ow_1,\ow_2\in D'$, $u_1,u_2 \in W$, $o(u_i)=\dd(\ow_i)=t(\ow_i)$, and $t(u_1)=t(u_2)$.
Relations (\ref{rel1}) (for $\cA_{D'}$) are satisfied for $\phi(s_{(\ow_1,u_1),(\ow_2,u_2)})$.
By Corollary~\ref{main1+} the homomorphism $\phi$ exists and is injective.
 Relation (\ref{rel1c}) for $\cA_D$
shows that each generator of $\cA_D$ is in the image of $\phi$.
Hence $\phi$ is an isomorphism.
\end{proof}

\begin{corollary}\label{27+}
For any $(D,\delta)$, $\cA_D$ is isomorphic to $\cA_{D'}$ for some $(D',\dd)$
with $\dd:D'\to A$ surjective.
\end{corollary}

\begin{proof}
By general hypothesis, $D$ is nonempty and $A$ is finite. Use Lemma~\ref{27} once to
replace $D$ with $D''$ so that $\#(D'')\ge\#(A)$, and use it again, in conjunction with Lemma~\ref{ends},
to construct the pair $(D',\dd)$.
\end{proof}

\begin{corollary}\label{27++}
For a fixed alphabet $A$ and fixed transition matrices $M_j$, the isomorphism class of
$\cA_D \otimes \cK$ is independent of $D$.
\end{corollary}

\begin{proof}
By Corollary~\ref{27+}, $\cA_D \cong \cA_{D'}$ for some $(D',\dd)$
with $\dd:D'\to A$ surjective. By Lemma~\ref{DN}, $\cA_{D' \times \NN} \cong \cA_{D'}\otimes\cK$.
Since $\dd$ is surjective, the decorating set $D' \times \NN$ is equivalent to the
decorating set $A \times \NN$: the inverse image of each $a\in A$
is countable. Thus $\cA_D\otimes\cK \cong \cA_{A \times \NN}$.
\end{proof}

Decorating sets other than $A$ and $A \times \NN$ arise naturally in the examples associated to
affine buildings.

\section{Construction of the algebra $\cA\otimes \cK$ as a crossed product}\label{ZRCP}

A vital tool in \cite{ck} was the expression of $\cO_A\otimes \cK$ as the crossed product
of an AF algebra by a $\ZZ$-action \cite[Theorem 3.8]{ck}.
The present section is devoted to an analogous result. In view of Lemma~\ref{DN},
this is done by establishing an isomorphism from $\cA_{A\times \NN}$ onto
the crossed product of an AF-algebra by a $\ZZ^r$-action. The AF-algebra will be isomorphic to
the algebra $\cF$ of Section~\ref{The AF subalgebra}, relative to the decorating set $A \times \NN$,
and the action of an element $k \in \ZZ^r$ will map the subalgebra $\cF_m$
onto $\cF_{m+k}$ for each $m \ge 0$.

Let a $C^*$-algebra $\cA'$ be defined just as $\cA$ is, with $D=A$ and with generators
$\{s'_{u,v}; u,v \in W \ \text{and} \ o(u) = o(v) \}$. The relations are the same as those
in (\ref{rel1}) except that in the sum (\ref{rel1c}), words are extended from the beginning.
The full relations are

\begin{subequations}\label{rel1'}
\begin{eqnarray}
{s'_{u,v}}^* &=& s'_{v,u} \label{rel1a'}\\
s'_{u,v}s'_{v,w}&=&s'_{u,w} \label{rel1b'}\\
s'_{u,v}&=&\displaystyle\sum_{\substack{w\in W;\s(w)=e_j,\\
t(w)=o(u)=o(v)}} s'_{wu,wv}   \label{rel1c'}\\
s'_{u,u}s'_{v,v}&=&0 \ \text{for} \ u,v \in  W_0, u \ne v. \label{rel1d'}
\end{eqnarray}
\end{subequations}

This may be thought of as using words with
extension in the negative direction. Alternatively,
replacing $M_j$ by the transpose matrix $M_j^t$ for each $j$ in the definition of $\cA$
results in an algebra isomorphic to $\cA'$.
Conditions (H0)--(H3) for the $M_j^t$ follow from the corresponding conditions for the $M_j$.
Consequently all the preceding results are valid for the algebra $\cA'$.

By Theorem~\ref{main1}, $\cA'$ is a simple separable $C^*$-algebra.
Let $\psi : \cA' \to \cB(\cH)$ be a nondegenerate representation of $\cA'$
on a \emph{separable} Hilbert space $\cH$. For $a \in A$,
let $\cH_a$ be the range of the projection $\psi(s'_{a,a})$. Then
$\cH = \bigoplus_{a\in A} \cH_a$. By Lemmas~\ref{f2} and \ref{nonzero+},
each $\cH_a$ is infinite dimensional.
Let $\cC = \bigoplus_{a\in A} \cK(\cH_a) \subset \cK(\cH)$.

For each $l \in \ZZ^r_+$ define a map $\alpha_l : \cC \to \cC$ by
\begin{equation}\label{alpha}\alpha_l(x) = \displaystyle\sum_{w \in W_l}
\psi(s'_{w,o(w)})x\psi(s'_{o(w),w}).
\end{equation}
Note that $\psi(s'_{w,o(w)})$ is a partial isometry with initial space $\cH_{o(w)}$
and final space lying inside $\cH_{t(w)}$.

\begin{lemma}\label{alpha1} Let $\alpha_l$ be as above.
\begin{enumerate}
\item $\alpha_l$ has image in $\cC$.
\item $\alpha_l$ is a $C^*$-algebra inclusion.
\item For $k,l \in \ZZ^r_+$, $\alpha_k\alpha_l=\alpha_{k+l}$.
\end{enumerate}
\end{lemma}

\begin{proof}
1. Clearly $\alpha_l(x) \in \cK(\cH)$. Moreover for fixed $w \in W$ with
$t(w) = a$, $\psi(s'_{w,o(w)})x\psi(s'_{o(w),w}) \in \cK(\cH_a)$.

2. Fix $w \in W_l$ with $o(w)=a$. Observe that $s'_{w,a}$ is a partial
isometry with initial projection $s_{a,a}$. Moreover for two different words
$w_1,w_2 \in W_l$, the range projections of $s'_{w_1,o(w_1)}$ and
$s'_{w_2,o(w_2)}$ are orthogonal. The result is now clear.

3. If $w_1 \in W_k$ and $w_2 \in W_l$ then according to
Lemmas~\ref{f2} and~\ref{5},
\begin{equation}\label{orth'}
s'_{w_1,o(w_1)}s'_{w_2,o(w_2)}=
\displaystyle\sum_
{\substack{
w_3 \in W_l \\
t(w_3)=o(w_1)
}}
s'_{w_3w_1,w_3}s'_{w_2,o(w_2)}
=
\begin{cases}
0 & \text{if $t(w_2) \ne o(w_1)$}\\
s'_{w_2w_1,o(w_2)} & \text{if $t(w_2)= o(w_1)$}.
\end{cases}
\end{equation}

Therefore
\begin{equation*}
\begin{split}
\alpha_k\alpha_l(x) & =
\displaystyle\sum_
{\substack{
w_1 \in W_k \\
w_2 \in W_l
}}
\psi(s'_{w_1,o(w_1)})\psi(s'_{w_2,o(w_2)})x\psi(s'_{o(w_2),w_2})\psi(s'_{o(w_1),w_1}) \\
& = \displaystyle\sum_
{\substack{
w_1 \in W_k \\
w_2 \in W_l \\
t(w_2)=o(w_1)
}}
\psi(s'_{w_2w_1,o(w_2)})x\psi(s'_{o(w_2),w_2w_1}) \\
& = \alpha_{k+l}(x).
\end{split}
\end{equation*}
\end{proof}

For each $m \in \ZZ^r$ let $\cC^{(m)}$ be an isomorphic copy of $\cC$, and for
each $l \in \ZZ^r_+$, let
$$\alpha_l^{(m)} : \cC^{(m)} \to \cC^{(m+l)}$$
be a copy of $\alpha_l$. Let $\cE = \varinjlim \cC^{(m)}$ be the direct limit
of the category of $C^*$-algebras with objects $\cC^{(m)}$ and morphisms $\alpha_l$ (\cite[Proposition 11.4.1]{kr}).
Then $\cE$ is an AF algebra. (See the discussion preceding Proposition~\ref{F}.)

If $x \in \cC$, let $x^{(m)}$ be the corresponding element of $\cC^{(m)}$. Then $x^{(m)}$ is identified with
$(\alpha_lx)^{(m+l)}$ for all $l \in \ZZ^r_+$. Define an action $\rho$ of $\ZZ^r$
on $\cE$ by $\rho(l)(x^{(m)})=x^{(m+l)}$.
Since $\ZZ^r$ is amenable, the full crossed product of $\cE$ by this action coincides with the reduced crossed product
\cite[Theorem 7.7.7]{ped},
and we denote it simply by $\cE \rtimes \ZZ^r$. The defining property of the crossed product says that there is
a unitary representation $m \mapsto U^m$ of $\ZZ^r$ into the multiplier algebra of
$\cE \rtimes \ZZ^r$ such that $\rho(l)(x^{(m)})=U^lx^{(m)}U^{-l}$, that is
\begin{equation}\label{covar}
U^lx^{(m)}=x^{(m+l)}U^{l}.
\end{equation}

\begin{theorem}\label{Z^rcross}
There exists an isomorphism $\phi : \cA_{A\times \NN} \to \cE \rtimes \ZZ^r$ where, moreover $\phi(\cF)=\cE$.
\end{theorem}

\begin{proof}
Let $D= A \times \NN$ and  $\delta(a,n)=a$.
Fix a map $\beta : D \to \cH$ so that $\{\beta(d); \delta(d)=a\}$ is an orthonormal
basis of $\cH_a$. For $\ow =(d,w) \in \oW_m$ define
$\beta(\ow) = \psi(s'_{w,o(w)})\beta(d)$,
in this way extending $\beta$ to a map $\beta : \overline W \to \cH$. Observe that for a fixed $w \in W$ with
$o(w)=a$, $\{\beta(d,w) ; d \in D, \delta(d)=a\}$ is an orthonormal basis for the range
of $s'_{w,w}$. Since, moreover the ranges of $\{s'_{w,w} ; w \in W_m\}$ are
pairwise orthogonal and sum to all of $\cH$, we see that $\{\beta(\ow) ; \ow \in \overline W_m\}$
is an orthonormal basis for $\cH$.

For $w \in W$ and $\ou =(d,u) \in \oW$, we have
\begin{equation*}
\psi(s'_{w,o(w)})\beta(\ou)  =
\psi(s'_{w,o(w)})\psi(s'_{u,o(u)})\beta(d)  =
\begin{cases}
0 & \text{if $o(w) \ne t(u)$}\\
\psi(s'_{uw,o(u)})\beta(d) & \text{if $o(w)=t(u)$},
\end{cases}
\end{equation*}
by equation (\ref{orth'}).
That is
\begin{equation*}
\psi(s'_{w,o(w)})\beta(\ou)  =
\begin{cases}
0 & \text{if $o(w) \ne t(u)$}\\
\beta(\overline uw) & \text{if $o(w)=t(u)$}.
\end{cases}
\end{equation*}

We will now define the map
$\phi : \cA_D \to \cE \rtimes \ZZ^r$.
For $\ou, \ov \in \oW$ with $t(\ou)=t(\ov)$ and $\s(\ou)=l$ and $\s(\ov)=m$ define
\begin{equation}\label{phi}
\phi(s_{\ov,\ou})= U^{m-l}\left(\beta(\ov) \otimes \overline{\beta(\ou)}\right)^{(l)}.
\end{equation}

We use the notation $\xi \otimes \overline \eta$ to denote the rank one operator
on a Hilbert space defined by $\zeta \mapsto \langle \zeta, \eta \rangle \xi$, so that
when $\xi$,$\eta$ have norm one, $\xi \otimes \overline \eta$ is a partial isometry
with initial projection $\eta \otimes \overline \eta$ and final projection
$\xi \otimes \overline \xi$. If the vectors $\xi$, $\eta$ vary through an orthonormal basis
for a Hilbert space then the operators $\xi \otimes \overline \eta$ form a system of matrix units
for the compact operators on that Hilbert space.
By equation (\ref{covar}) we have
\begin{equation}\label{covar+}
\phi(s_{\ov,\ou})= \left(\beta(\ov) \otimes \overline{\beta(\ou)}\right)^{(m)}U^{m-l}.
\end{equation}

We show that the partial isometries $\phi(s_{\ov,\ou})$ satisfy the relations (\ref{rel1}).
Relation (\ref{rel1d}) is immediate from the definition of $\beta(\ow)$, since
if $\ow \in \oW_0$ then
$\phi(s_{\ow,\ow})= \left(\beta(\ow) \otimes \overline{\beta(\ow)}\right)^{(0)}$.

Relation (\ref{rel1a}) is satisfied since, by (\ref{covar+}),
\begin{equation*}
\phi(s_{\ov,\ou})^* = \left(\left(\beta(\ov) \otimes
\overline{\beta(\ou)}\right)^{(m)}U^{m-l}\right)^*
= U^{l-m}\left(\beta(\ou) \otimes \overline{\beta(\ov)}\right)^{(m)}= \phi(s_{\ou,\ov}).
\end{equation*}
If $t(\ow)=t(\ov)$ and $\s(\ow)=n$, then
\begin{equation}
\begin{split}
\phi(s_{\ow,\ov})\phi(s_{\ov,\ou}) &=
U^{n-m}\left(\beta(\ow) \otimes \overline{\beta(\ov)}\right)^{(m)}U^{m-l}\left(\beta(\ov) \otimes
\overline{\beta(\ou)}\right)^{(l)} \\
&=
U^{n-m}U^{m-l}\left(\beta(\ow) \otimes \overline{\beta(\ov)}\right)^{(l)}\left(\beta(\ov) \otimes
\overline{\beta(\ou)}\right)^{(l)} \\
&=
U^{n-l}\left(\beta(\ow) \otimes \overline{\beta(\ou)}\right)^{(l)} \\
&=
\phi(s_{\ow,\ou}).
\end{split}
\end{equation}
Thus (\ref{rel1b}) is satisfied. Finally (\ref{rel1c}) is a consequence of the following calculation.
\begin{equation*}
\begin{split}
\phi(s_{\ov,\ou}) &= U^{m-l}\left(\beta(\ov) \otimes \overline{\beta(\ou)}\right)^{(l)} \\
&= U^{m-l}\left(\alpha_k(\beta(\ov) \otimes \overline{\beta(\ou)})\right)^{(l+k)} \\
&= U^{m-l}\displaystyle \sum_{w \in W_k}
\left(\psi(s'_{w,o(w)})\beta(\ov) \otimes \overline{\psi(s'_{w,o(w)})\beta(\ou)}\right)^{(l+k)} \\
&= U^{m-l}\displaystyle \sum_{\substack{w \in W_k \\ o(w)=t(\ou)}}
\left(\beta(\ov w) \otimes \overline{\beta(\ou w)}\right)^{(l+k)} \\
&= \displaystyle \sum_{\substack{w \in W_k \\ o(w)=t(\ou)}}
\phi(s_{\ov w,\ou w}).
\end{split}
\end{equation*}

All the relations (\ref{rel1}) are satisfied by the partial isometries $\phi(s_{\ov,\ou})$.
Therefore by Lemma~\ref{6}, $\phi$ defines a *-homomorphism. Clearly $\phi$ is not the zero map; hence it
is an isometry, by Theorem~\ref{main1}. It only remains to show that $\phi$ is
onto.

Fix $m, n \in \ZZ^r$ so that $m, m+n \ge 0$. For $\ou \in \oW_m$ and $\ov \in \oW_{m+n}$
with $t(\ou)=t(\ov)=a$, we have
\begin{equation}
\phi(s_{\ov,\ou})= U^n\left(\beta(\ov) \otimes \overline{\beta(\ou)}\right)^{(m)}.
\end{equation}

As the sets
$\{\beta(\ov) ;\ \ov \in \oW_{m+n}, t(\ov)=a\}$ and
$\{\beta(\ou) ;\ \ou \in \oW_m, t(\ou)=a\}$
are bases for $\cH_a$, the image of $\phi$ contains a dense subset of $U^n\left(\cK(\cH_a)\right)^{(m)}$.
Therefore the image of $\phi$ contains $U^n\left(\cK(\cH_a)\right)^{(m)}$, for each $a \in A$.
It therefore contains $U^n\cC^{(m)}$.

Also, for any $k \ge 0$,
$$\phi(\cA_{A\times \NN}) \supseteq U^n\cC^{(m)} \supseteq U^n \alpha_k^{(m-k)}\cC^{(m-k)}= U^n\cC^{(m-k)}.$$
It follows that $\phi(\cA_{A\times \NN}) = \cE \rtimes \ZZ^r$.
It is clear from the definitions that $\phi(\cF_m)=\cC^{(m)}$ and that $\phi(\cF)=\cE$.
\end{proof}

\begin{corollary}\label{Z^rcross+} $\cA \otimes\cK \cong \cE \rtimes \ZZ^r$.
\end{corollary}

\begin{proof} This follows immediately from Theorem~\ref{Z^rcross}, Lemma~\ref{DN} and Corollary~\ref{27++}.
\end{proof}

\begin{corollary}\label{nuclear} $\cA$ is nuclear.
\end{corollary}

\begin{proof}
This follows because the class of nuclear $C^*$-algebras is closed under stable isomorphism and
crossed products by amenable groups, and contains the AF algebras.
\end{proof}

\begin{remark}\label{previous} Suppose that $D$ is finite, so that $\cA$ is unital by
Remark~\ref{finitedecorate}.
Then it has been established that the separable unital $C^*$-algebra $\cA$ is
simple (Theorem~\ref{main1}), nuclear (Corollary~\ref{nuclear}) and purely infinite
(Proposition~\ref{purely_infinite}). Corollary~\ref{Z^rcross+} also shows that $\cA$ belongs to the
bootstrap class ${\mathcal N}$, which contains the AF-algebras and is closed under
stable isomorphism and crossed products by $\ZZ$. Thus $\cA$ satisfies the Universal Coefficient Theorem
\cite[Theorem 23.1.1]{bl}. The work of \textsc{E. Kirchberg} and \textsc{C. Phillips}
 \cite{k,k'},\cite{p} therefore shows that $\cA$  is classified by its K-groups.
\end{remark}

\section{Boundary actions on affine buildings of type~$\wt A_2$}\label{boundary-algebra}

Let $\cB$ be a locally finite thick affine building of type~$\wt A_2$.
This means that $\cB$ is a chamber system consisting of vertices, edges and triangles
(\emph{chambers}).  An \emph{apartment} is a subcomplex of $\cB$ isomorphic to the Euclidean
plane tesselated by equilateral triangles.  A \emph{sector} (or \emph{Weyl chamber}) is a
$\frac{\pi}{3}$-angled sector made up of chambers in some apartment.
Two sectors are  \emph{equivalent} (or parallel) if their
intersection contains a sector. We refer to \cite{bro,g,ron} for the theory of buildings.
Shorter introductions to the theory are provided by \cite{bro',ca,st}.

The boundary $\Omega$ is defined to be the set of equivalence classes of sectors in $\cB$.
In $\cB$ we fix some vertex $O$, which we assume to have type $0$.
For any $\omega \in \Omega$ there is a unique sector $[O,\omega)$ in the
class $\omega$ having base vertex $O$ \cite[Theorem 9.6]{ron}.
The boundary $\Omega$ is a totally disconnected compact Hausdorff space with a base for the
topology given by sets of the form
$$
\Omega(v) = \left \{ \omega \in \Omega : [O,\omega) \ \text{ contains } v \right \}
$$
where $v$ is a vertex of $\cB$ \cite[Section 2]{cms}. We note that if $[O,\omega) \ \text{ contains } v$
then $[O,\omega)$ contains the parallelogram $\conv(O,v)$.

Let $\G$ be a group of type rotating automorphisms of $\cB$ that acts freely on the vertex set with finitely many orbits. See \cite{cmsz} for a discussion and examples in
the case where $\G$ acts transitively on the vertex set. There is a natural induced
action of $\G$ on the boundary $\Omega$ and we can form the universal crossed product algebra $C(\Om)\rtimes \G$ \cite{ped}. The purpose of this section is to identify $C(\Om)\rtimes \G$ with an algebra of the form $\cA$.

Let $\fa$ be a Coxeter complex of type~$\wt A_2$, which we
shall use as a model for the apartments of $\cB$.  Each vertex of
$\fa$ has type $0$,$1$, or $2$. Fix as the origin in $\fa$ a vertex of
type $0$.  Coordinatize the vertices by $\ZZ^2$ by choosing a fixed
sector in $\fa$ based at the origin and defining the positive
coordinate axes to be the corresponding sector panels (\emph{cloisons
de quartier}).  The coordinate axes are therefore given by two of the
three walls of $\fa$ passing through the origin.  Let $\ft$ be a model
\emph{tile} in $\fa$ and let $\fp_m$ be a model parallelogram in $\fa$
of \emph{ shape} $m = (m_1,m_2)$, as illustrated in
Figure~\ref{models}.  As per Figure~\ref{models}, assume that $\ft$
and $\fp_m$ are both based at $(0,0)$.  Thus $\ft$ is the model
parallelogram of shape $(0,0)$.

\refstepcounter{picture}
\begin{figure}[htbp]\label{models}
\hfil
\beginpicture
\setcoordinatesystem units <0.5cm,0.866cm>  point at -6 0  
\setplotarea x from -5 to 5, y from -3.5 to 4.5            
\put{$_{(0,0)}$}  [t]   at  1   -2.2
\put{$_{(m_1+1,m_2+1)}$}  [b]   at -1    4.2
\put{$_{(m_1+1,0)}$}  [l]   at  3.2  0
\put{$_{(0,m_2+1)}$}  [r]   at -3.2    2
\setlinear \plot   1    -2        3   0 /
\setlinear \plot  -3     2      -1   4 /
\setlinear \plot  -1     4       3   0 /
\setlinear \plot  -3     2        1  -2 /
\put {A model parallelogram $\fp_m$.}   at   0 -3.5
\setcoordinatesystem units <0.5cm,0.866cm> point at 6 0   
\setplotarea x from -5 to 5, y from -3.5 to 4.5         
\put{$_{(0,0)}$}     [t]   at  0  -0.2
\put{$_{(1,1)}$} [b] at  0 2.2
\put{$_{(0,1)}$} [r]  at -1.2 1
\put{$_{(1,0)}$}  [l] at  1.2 1
\putrule from -1 1 to 1 1
\setlinear \plot -1 1  0 0  1 1  /
\setlinear \plot -1 1  0 2  1 1 /
\put {The model tile $\ft$.}   at   0 -2
\endpicture
\hfil
\caption{}
\end{figure}

Let $\fT$ denote the set of type rotating isometries $i :\ft \to\cB$,
and let $A = \G\backslash \fT$. We will use the set $A$ as an alphabet
to define an algebra $\cA$.  Let $\fP_m$ denote the set of type
rotating isometries $p:\fp_m \to \cB$, and let $\fW_m = \G\backslash
\fP_m$.  Let $\fP = \bigcup_m \fP_m$ and $\fW = \bigcup_m \fW_m$.

If $p \in \fP_m$, then define $t(p) : \ft \to \cB$ by $t(p)(l) = p(m+l)$. Then $t(p)$ is a type rotating isometry
such that $t(p)(\ft)$ lies in $p(\fp_m)$
 with $t(p)(1,1)=p(m_1+1,m_2+1)$, as illustrated in Figure~\ref{t(p)}. Thus $t(p) \in \fT$.
Similarly $o(p) : \ft \to \cB$ is defined by $o(p)=p\vert_{\ft}$.

\refstepcounter{picture}
\begin{figure}[htbp]
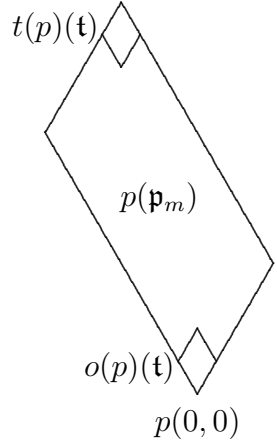
\label{t(p)}
\hfil
\centerline{
\beginpicture
\setcoordinatesystem units <0.5cm,0.866cm>    
\put{$p(0,0)$}  [t] at  1 -2.2
\put{$p(\fp_m)$}  at 0  1
\put{$t(p)(\ft)$}  [r] at  -1.6 3.6
\put{$o(p)(\ft)$}  [r] at  0.4 -1.6
\setplotarea x from -6 to 6, y from -3 to 4         
\setlinear
\plot 1  -2    3  0   -1  4     -3 2   1 -2   /
\plot -1.5 3.5  -1 3  -0.5 3.5 /
\plot   0.5 -1.5   1 -1   1.5 -1.5 /
\endpicture
}
\hfil
\caption{The initial and terminal tiles.}
\end{figure}

The matrices $M_1$, $M_2$ with entries in $\{0,1\}$
are defined as follows.
If $a,b \in A$, say that $M_1(b,a)=1$ if and only if there exists $p \in \fP_{(1,0)}$
such that $a=\G o(p)$ and $b=\G t(p)$.
Similarly, if $c \in A$ then
$M_2(c,a)=1$ if and only if there exists $p \in \fP_{(0,1)}$
such that $a=\G o(p)$ and $c=\G t(p)$.
The definitions are illustrated in Figure~\ref{tiles}, for suitable
representative isometries $i_a$,$i_b$,$i_c$ in $\fT$.

\refstepcounter{picture}
\begin{figure}[htbp]
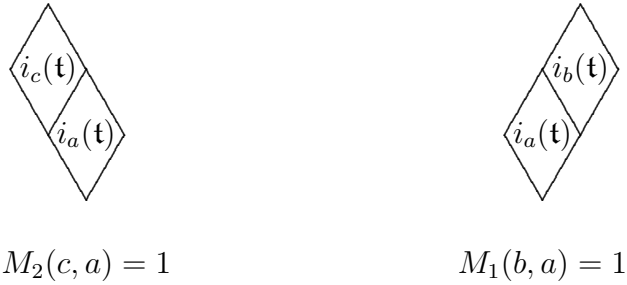
\label{tiles}
\hfil
\centerline{
\beginpicture
\setcoordinatesystem units <0.5cm,0.866cm>  point at -6 0 
\setplotarea x from -5 to 5, y from -1 to 3.3         
\put{$i_b(\ft)$}   at  1  2
\put{$i_a(\ft)$}   at  0 1
\put{$M_1(b,a)=1$}   at  0 -1
\setlinear \plot -1 1  0 0  1 1  /
\setlinear \plot -1 1  0 2  1 1 /
\setlinear \plot 1 1  2 2  1 3  0 2  /
\setcoordinatesystem units <0.5cm,0.866cm>  point at 6 0 
\setplotarea x from -5 to 5, y from -1 to 3.3         
\put{$i_c(\ft)$}   at  -1  2
\put{$i_a(\ft)$}   at  0 1
\put{$M_2(c,a)=1$}   at  0 -1
\setlinear \plot -1 1  0 0  1 1  /
\setlinear \plot -1 1  0 2  1 1 /
\setlinear \plot -1 1  -2 2  -1 3  0 2  /
\endpicture
}
\hfil
\caption{Definition of the transition matrices.}
\end{figure}

In order to apply the general results we need to verify that
conditions (H0)-(H3) are satisfied. We address this question in the
subsection~\ref{conditionsH0-H4}.  Until further notice we simply
impose the following

\begin{quote}
{\bf ASSUMPTION}: Conditions {\rm (H0)-(H3)} are satisfied.
\end{quote}

We can now define the set of words $W_m$ of shape $m \in \ZZ_+$ based
on the alphabet $A$ and the transition matrices $M_1$,$M_2$, as in
Section 1.  There is also a natural map $\alpha : \fP \to W$, defined
as follows. Given $p \in \fP_m$, construct $\alpha(p)=w$ according to
the following procedure. For each $n \in \ZZ^2_+$ with $0 \le n \le
m$, let $w(n)=\G p_n$ where $p_n \in \fT$ is defined by
$p_n(l)=p(n+l)$, for $(0,0) \le l \le (1,1)$.  Since the translation
$l \mapsto n+l$ is a type-rotating isometry of $\fa$, it follows that
$p_n$ is type rotating, hence an element of $\fT$.  Passing to the
quotient by $\G$ gives a well defined map $\alpha : \fW_m \to W_m$ and
hence a map $\alpha : \fW \to W$.  The use of $\alpha$ for the
different maps should not cause confusion.  If $a= \G i \in A = \fW_0$
then $\alpha(a)=a$, so it is clear that $o(\alpha(p))=\G o(p)$ and
$t(\alpha(p))= \G t(p)$.

\begin{lemma}\label{frWm}
The map $\alpha$ is a bijection from $\fW_m$ to $W_m$ for each $m \in \ZZ^2_+$.
\end{lemma}

\begin{proof}
Suppose that $\alpha(p)=\alpha(p')$.
Then $\G p_n = \G p'_n$ for $0 \le n \le m$. For each $n$ there exists $\g_n$ so that $\g_np_n=p_n'$.
Since $\G$ acts freely on the vertices, each $\g_n$ is uniquely determined.
Moreover, since $p_n$ and $p_{n+e_j}$ share a pair of vertices in their image, it must be true that
$\g_n=\g_{n+e_j}$. By induction, the $\g_n$ have a common value, $\g$, and $\g p=p'$.
Thus $\G p = \G p'$.

It remains to show that $\alpha$ is surjective. Let $w \in
W_m$. Choose a path $(n_0,\dots,n_k)$ of points in $\ZZ^2$ so that
$n_0=0, n_k=m$, and each difference $n_{j+1}-n_j$ is either $e_1$ or
$e_2$.  Choose representative type rotating isometries $i_0,i_1, \dots
,i_k$ from $\ft$ into $\cB$ with $\G i_r=w(n_r)$ so that $i_r(\ft)$
and $i_{r+1}(\ft)$ are adjacent tiles in $\cB$, according to the two
possibilities in Figure~\ref{tiles}. This defines a gallery
$\fG=\{i_0(\ft),i_1(\ft), \dots ,i_k(\ft)\}$.  Let $\fG_0=\{C_0,C_1,
\dots,C_k\}$ be the corresponding gallery in $\fp_m$ with
$C_0=\ft$. It is clear from Figure~\ref{inductive step} that $\fG_0$
is a minimal gallery in $\fa$.  (The elements of $\fG$ and $\fG_0$ are
tiles rather than chambers, but this is immaterial.) The obvious map
$p : \fG_0 \to \fG$ is a strong isometry (preserves generalized
distance). Therefore $\fG$ is contained in an apartment and $p$
extends to a strong isometry $p$ from the convex hull
$\conv(\fG_0)=\fp_m$ into that apartment. See \cite[p.\ 90,\
Theorem]{bro} and \cite[Appendix B]{bro'}.  Thus $p \in \fP_m$ and
since $\alpha(p)$ agrees with $w$ at $n_0,n_1,\dots,n_k$, (H1) implies
that $\alpha(p)=w$.
\end{proof}

\refstepcounter{picture}
\begin{figure}[htbp]\label{inductive step}
\hfil
\centerline{
\beginpicture
\setcoordinatesystem units <0.5cm,0.866cm>    
\setplotarea x from -6 to 6, y from -6 to 2         
\setlinear
\plot 1  2  -2.5 -1.5  1 -5  4.5 -1.5  1 2  /
\plot 1 -2  0.5 -1.5  1 -1  1.5 -1.5  1 -2 /
\plot 1 -5 1.5 -4.5 1 -4 0.5 -4.5 /
\plot 1 2  1.5 1.5  1 1  0.5 1.5 /
\linethickness=4pt
\setdashes
\plot  1.25 -1.25   2.5 0    1.25  1.25  /
\plot  0.75 -1.75  -0.5 -3  0.75 -4.25  /
\endpicture
}
\hfil
\caption{A minimal gallery in in $\fa$.}
\end{figure}

Let $\ofW_m$ denote the set of type rotating isometries $p :\fp_m \to
\cB$ such that $p(0,0)=O$ and let $\ofW = \bigcup_m \ofW_m$.  Let $D$
denote the set of type-rotating isometries $d : \ft \to \cB$ such that
$d(0,0)=O$.  Let $\delta : D \to A$ be given by $\delta(d)=\G d$. The
map $\delta$ is injective since $\G$ acts freely on the vertices of
$\cB$. Moreover $\delta$ is surjective if and only if $\G$ acts
transitively on the vertices. Define $\oa : \ofW \to \oW$ by
$\oa(p)=\left(o(p), \alpha(\G p)\right)$. (Recall that $o(p)=
p\vert_{\ft}$.)

\begin{lemma} \label{oalpha}
The map $\oa$ is a bijection from $\ofW_m$ onto $\oW_m$ for each $m \in \ZZ^2_+$.
\end{lemma}

\begin{proof}
If $\oa(p_1)=\oa(p_2)$ then $o(p_1)=o(p_2)$; moreover $\G p_1=\G p_2$,
by Lemma~\ref{frWm}.  Since $\G$ acts freely on the vertices, it
follows that $p_1=p_2$.  Therefore $\oa$ is injective.

To see that it is surjective, let $\ow = (d,w) \in \oW_m$, where $w
\in W_m$ and $d \in D$. By Lemma~\ref{frWm}, there exists $p \in
\fP_m$ such that $\alpha( \G p)=w$. Then
\begin{equation*}
\G d=\delta (d)=o(w)=o(\alpha(\G p))= \G o(p).
\end{equation*}
Replacing $p$ by $\g p$ for suitable $\g \in \G$ ensures that $o(p)=d$
and hence $p\in \ofW_m$ and $\oa (p)=\ow$.
\end{proof}

If $p \in \ofW$ then $\conv(O,t(p)(\ft))=\conv(O,p(m_1+1,m_2+1))$ and we introduce the notation
\begin{equation*}
\Omega(p)=\Omega(p(m_1+1,m_2+1)) = \left \{ \omega \in \Omega ; t(p) \subset [O,\omega) \right \}
=\left \{ \omega \in \Omega ; p(\fp_m) \subset [O,\omega) \right \}.
\end{equation*}

Let $i\in \fT$, that is, suppose that $i :\ft \to\cB$ is a type
rotating isometry. Let
\begin{equation*}
\Omega(i)=\left \{ \omega \in \Omega ; i(\ft) \subset [i(0,0),\omega) \right \},
\end{equation*}
those boundary points represented by sectors which originate at $i(0,0)$ and contain $i(\ft)$.
Clearly  $\Omega(\g i)=\g\Omega(i)$. For $p \in \ofW_m$ we have $\Omega(p)=\Omega(t(p))$.
Indeed, any sector originating at $t(p)(0,0)$ and containing $t(p)(\ft)$ extends to a sector originating
at $O$ and containing $p(\fp_m)$.

Fix $\ow_1, \ow_2 \in \oW$ with $t(\ow_1)=t(\ow_2)=a \in A$. Let
$p_1=\oa^{-1}(\ow_1)$ and $p_2=\oa^{-1}(\ow_2)$.
Let $\g \in \G$ be the unique element such that $\g t(p_1)=t(p_2)$.
Define a homomorphism $\phi: \cA \to C(\Om)\rtimes \G$ by
\begin{equation}\label{isomorphism}
\phi(s_{\ow_2,\ow_1})=\g \Ind _{\Om(p_1)}= \Ind _{\Om(p_2)}\g.
\end{equation}

Note that by Lemma~\ref{6} this does indeed define a *-homomorphism of
$\cA$ because the operators of the form $\phi(s_{\ow_2,\ow_1})$
are easily seen to satisfy the relations (\ref{rel1}).  We now prove
that $\phi$ is an isomorphism from $\cA$ onto $C(\Om)\rtimes \G$
(Theorem~\ref{main2} below).  For this some preliminaries are necessary.

\begin{lemma} \label{c1}
For any $m \in \ZZ^2_+$, $\Ind = \displaystyle \sum_{p \in \ofW_m} \Ind _ {\Omega(p)}$.
\end{lemma}

\begin{proof}
This follows from the discussion in \cite[Section 2]{cms}.
\end{proof}

\begin{lemma} \label{c2}
The linear span of $\{ \Ind _ {\Omega(p)} ; p \in \ofW \}$ is dense in $C(\Om)$.
\end{lemma}

\begin{proof}
This follows because the sets $\Omega(p)$ for  $p \in \ofW$ form a basis for the topology
of $\Om$ \cite[Section 2]{cms}.
\end{proof}

\begin{lemma}\label{c3}
Let $p \in \fP_m$ where $m = (m_1,m_2) \in \ZZ^2_+$.
Let $x = p(0,0)$ and $y= p(m_1+1,m_2+1)$. Let $y'$ be another vertex of $\cB$
whose graph distance to $y$ in the $1$-skeleton of $\cB$ equals $n$. Suppose that
$n \le m_1, m_2$. Then $\conv(x,y')$ contains
$p(\fp_{(m_1-n,m_2-n)})$.
\end{lemma}

\begin{proof}
Induction reduces us to the case $n=1$. There is some apartment
containing~$x$ and the edge from~$y$ to~$y'$. In Figure~\ref{c3'} we
show one of the six possible positions for~$y'$.

\refstepcounter{picture}
\begin{figure}[htbp]
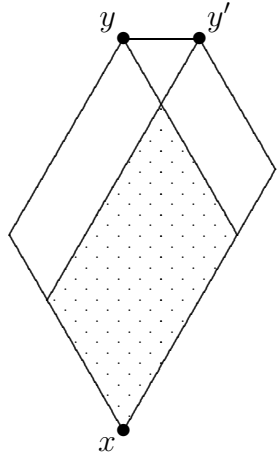
\label{c3'}
\hfil
\centerline{
\beginpicture
\setcoordinatesystem units <0.5cm,0.866cm> point at -4 1   
\setplotarea x from -3 to 3, y from -2.5 to 4.5         
\setsolid
\put {$\bullet$} at -1 4
\put {$\bullet$} at 1 4
\put {$\bullet$} at -1 -2
\put {$y$} [r,b] at -1.2 4.1
\put {$y'$} [l,b] at 1.2 4.1
\put {$x$} [r,t] at -1.2 -2.1
\putrule from -1 4 to 1 4
\plot -1  -2    -3  0   1  4     3 2   -1 -2   /
\plot -3 0   -4 1   -1 4   2 1  /
\setshadegrid span <4pt>
\vshade -3 0 0 <,z,,>  -1 -2 2  <z,z,,>  0 -1 3 <z,,,>  2 1 1  /
\endpicture
}
\hfil
\caption{Relative positions of $x$,$y$ and $y'$ in an apartment.}
\end{figure}

Now $\conv(x,y')$ is the image of some $p' \in \fP_{(m_1+1,m_2-1)}$.
It is evident in this case (and in the other five cases) that
$p'(\fp_{(m_1-1,m_2-1)})=p(\fp_{(m_1-1,m_2-1)})$.
\end{proof}

\begin{corollary}\label{9B}
Let $p \in \fP_m$ for $m=(m_1,m_2)$. Let $x=p(m_1+1,m_2+1)$ and $y=p(0,0)$.
Let $y'$ be a third vertex of $\cB$ at distance $n$ to $y$.
Suppose $n \le m_1,m_2$. Then $\conv(x,y')$ is the image of some $p' \in \fP$,
where $p'(0,0)=y'$ and $t(p')=t(p)$.
\end{corollary}

\refstepcounter{picture}
\begin{figure}[htbp]\label{9B'}
\hfil
\centerline{
\beginpicture
\setcoordinatesystem units <0.5cm,0.866cm>    
\setplotarea x from -6 to 6, y from -3 to 4.3         
\put{$y$}  [t] at  1  -2.2
\put{$x$}  [l] at  -0.8  4.1
\put{$y'$}  [t] at  -2.5  -1.7
\put{$t(p)(\ft)$}  [r] at  -1.6 3.6
\setlinear
\plot 1  -2    3  0   -1  4     -3 2   1 -2   /
\plot -3 2    -4.5 0.5   -2.5 -1.5    1 2   /
\plot -1.5 3.5  -1 3  -0.5 3.5 /
\endpicture
}
\hfil
\caption{}
\end{figure}

\begin{proof}
See Figure~\ref{9B'}. According to Lemma~\ref{c3}, the convex hull
$\conv(x,y')$ contains $t(p)(\ft)$.  Because $\conv(x,y')$
contains a chamber, it must be the image of a unique $p' \in \fP$ with
$p'(0,0)=y'$. Since $t(p)(\ft)$ lies in the image of $p'$, namely
$\conv(x,y')$, and $t(p)(1,1)=x$, it follows that $t(p')~=~t(p)$.
\end{proof}

\begin{theorem}\label{main2}
The map $\phi$ is an isomorphism from $\cA$ onto $C(\Om)\rtimes \G$.
\end{theorem}

\begin{proof} Since $\cA$ is simple, $\phi$ is injective.
If $p \in \oW$ then $\Ind _{\Omega(p)}=\phi(s_{\ow,\ow})$, where $\ow=\oa(p)$ and so
Lemma~\ref{c2} shows that the range of $\phi$ contains $C(\Om)$.
It remains to show that it contains $\G$.
Fix $\g \in \G$. Choose $m=(m_1,m_2)$ so that $d(O,\g^{-1}O) \le m_1, m_2$.  Now
$\g = \g \Ind = \displaystyle \sum_{p \in \ofW_m} \g \Ind _ {\Omega(p)}$.
For $p \in \ofW_m$ we claim that $\g t(p) = t(p')$ for some $p' \in \ofW$. Hence
$\g \Ind _ {\Omega(p)} = \phi \left(s_{\alpha(p'),\alpha(p)}\right)$.
This shows the range of $\phi$ contains $\G$ and hence is surjective.

To prove the claim above, apply Corollary~\ref{9B}
 to find $p'' \in \fP$ with $p''(0,0)= \g^{-1}(O)$
and $t(p'')=t(p)$. Let $p' = \g p''$. Then $p'(0,0)=O$, so $p' \in \ofW$
and $t(p')=\g t(p'')=\g t(p)$.
\end{proof}

\begin{remark}
It follows from Theorem~\ref{main2} and Remark~\ref{previous} that
$C(\Om)\rtimes \G$ is simple, nuclear and purely infinite. Simplicity
and nuclearity had previously been proved in \cite{rs} under the
additional assumption that $\G$ acts transitively and in a type
rotating manner on the vertices of $\cB$. From simplicity it follows
that $C(\Om)\rtimes\G$ is isomorphic to the reduced crossed product
$C(\Om)\rtimes_r \G$.  See also \cite {an,qs} for general conditions
under which the full and reduced crossed products coincide.  Their
results apply, for example when $\G$ is a lattice in a linear group.
\end{remark}

\subsection{Conditions {\rm (H0)-(H3)} for affine buildings of type
$\wt A_2$}\label{conditionsH0-H4}

Continue the notation and terminology used above, assuming throughout that
$\G$~acts on~$\cB$ via type rotating automorphisms.  We prove that
conditions (H0), (H1), and (H3) are satisfied so long as $\G$~acts
freely and with finitely many orbits on the vertices of~$\cB$.
Moreover, if $\cB$ is the building of $G=\PGL_3(\KK)$, where $\KK$ is
a nonarchimedean local field of characteristic zero and $\G$ is a
lattice in $\PGL_3(\KK)$ we prove that (H2) holds as well.  There are
several concrete examples in \cite{cmsz} where all these hypotheses
are satisfied.

\begin{proposition}\label{M_1M_2} Suppose that $\G$~acts freely and
with finitely many orbits on the vertices of~$\cB$.  Then the matrices
$M_1$, $M_2$ of the previous section satisfy conditions
{\rm(H0)},{\rm(H1)}, and {\rm(H3)}.
\end{proposition}

\begin{proof}
By definition $M_1$ and~$M_2$ are $\{0,1\}$-matrices.  To say that
they are nonzero is to say that $\fP_{(1,0)}$ and~$\fP_{(0,1)}$ are
nonempty, which the are.  This proves~(H0).

Fix any nonzero $j\in\ZZ^2$.  We will construct $w\in\fW$ which is not
$j$-periodic.  Choose $m\in\ZZ^2$ large enough so that inside $\fp_m$
one can find a minimal gallery of chambers, $(C_0,C_1,\dots,C_l)$ so
that $C_l$~is the $j$-translate of~$C_0$.  Write $\tau:C_0\to C_l$ for
the identification by translation of the two chambers.

Construct an isometry~$p$ from this minimal gallery to~$\cB$ by
defining successively $p|_{C_0}$, $p|_{C_1}$, etc.  Since the building
is thick, one has at least two choices at each step.  Once
$p|_{C_{l-1}}$~is fixed, no two of the choices for $p|_{C_l}$ can be in
the same $\G$-orbit, since $\G$ acts freely on the vertices of~$\cB$.
Therefore, one may choose~$p$ so that $p|_{C_0}$ and $p\circ\tau$ are
in different $\G$-orbits.  Now extend $p$ to an isometry
$p:\fp_m\to\cB$.  The element of~$W$ associated to~$p$, that is
$\alpha(\Gamma p)$, is not $j$-periodic.  This proves~(H3).

Condition~(H1c) is vacuous for~$r=2$.  Consider the configuration of
Figure~\ref{unique_extension}.  Given the tiles $a$, $b$, and~$c$,
there is exactly one tile~$d$ which completes the picture.  Since $a$,
$b$, and~$c$ make up a minimal gallery, this follows by the same
argument used in proving Lemma~\ref{frWm}.  Translating this fact to
matrix terms, we have that if $(M_2M_1)(c,a)>0$ then
$(M_1M_2)(c,a)=1$.  Likewise, if $(M_1M_2)(c,a)>0$, then
$(M_2M_1)(c,a)=1$.  Conditions~(H1a) and~(H1b) follow, and by
Lemma~\ref{3}, so does condition~(H1).
\end{proof}

\refstepcounter{picture}
\begin{figure}[htbp]
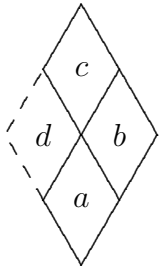
\label{unique_extension}
\hfil
\centerline{
\beginpicture
\setcoordinatesystem units <0.5cm,0.866cm>  
\setplotarea x from -5 to 5, y from 0 to 4         
\put{$a$}   at  0 1
\put{$b$}   at  1  2
\put{$c$}   at  0 3
\put{$d$}   at  -1  2
\setlinear \plot -1 1  0 0   2 2  0 4  -1 3  0 2  1 3 /
\plot -1 1  0 2  1 1 /
\setdashes \plot  -1 1  -2 2  -1 3 /
\endpicture
}
\hfil
\caption{Uniqueness of extension.}
\end{figure}

It is not much harder to prove~(H1) directly, bypassing conditions
(H1a)-(H1c).  It remains to prove condition~(H2).  The next result and its
corollary prove a strong version of condition~(H2).

\begin{theorem}\label{H(3)}
Let $\cB$ be the building of $G=\SL_3(\KK)$, where $\KK$ is a local
field of characteristic zero.  Let $\G$ be a lattice in $\SL_3(\KK)$
which acts freely on the vertices of $\cB$ with finitely many orbits.
Let the alphabet $A=\G\backslash \fT$ and the transition matrices
$M_1$,$M_2$ be defined as at the beginning of
Section~\ref{boundary-algebra}. Then for each $i=1,2$, the directed
graph with vertices $a \in A$ and directed edges $(a,b)$ whenever
$M_i(b,a) =1$ is irreducible.
\end{theorem}

\begin{proof}
We use an idea due to \textsc{S. Mozes} \cite[Proposition 3]{moz'}.
Fix a model half-infinite strip $\fs$ of tiles in $\fa$ based at $(0,0)$ and let $\fs_k= \fp_{(k,0)}$
 be the initial segment consisting of $k+1$ tiles.
\refstepcounter{picture}
\begin{figure}[htbp]\label{H(1)}
\hfil
\centerline{
\beginpicture
\setcoordinatesystem units <0.5cm,0.866cm>  
\setplotarea x from -6 to 6, y from 0 to 1         
\putrule from  -5 1 to  6 1
\putrule from  -6 0  to  5 0
\setlinear \plot  -5 1   -6 0 /
\plot -3 1  -4 0 /
\plot -1 1  -2 0 /
\plot  1 1   0 0 /
\plot  3 1   2 0 /
\plot  5 1   4 0 /
\endpicture
}
\hfil
\caption{The strip $\fs$.}
\end{figure}

Let $\fS_0$ [respectively $\fS_k$, $k\ge1$] be the set of type-preserving isometric embeddings $s$
of $\fs$ [respectively $\fs_k$] into the building $\cB$.
Thus $\fS_k$ is the subset of type-preserving maps in $\fP_{(k,0)}$ and if $s\in\fS_k$, then  $o(s)$ and
$t(s)$ are defined as before.
Also, if $s\in\fS$ define $o(s)=s\vert_{\ft}$ and let $s_k=t(s\vert_{\fs_k})$ for $k \ge 0$.

It is desired to prove that given $a,b \in A$,
there exists  $k \in \ZZ_+$ and $s \in \fS_k$ such that $a=\G o(s)$ and $b=\G t(s)$.
The group $G$ acts transitively on the set of apartments of $\cB$ \cite[Section 5]{st}. Moreover
the stabilizer of an apartment acts transitively on the set of sectors of the apartment.
It follows that $G$ acts transitively on $\fS_0$.
Therefore $\fS_0=G/H_0$ for some subgroup $H_0$.  In fact
$$
H_0=\{
      \begin{pmatrix} a_1 &  b_2 & b_3 \\ 0 & a_2 & c \\ 0 & d & a_3
      \end{pmatrix} \in G ;
         \text{ $|a_j|=1$, $|b_j|,|c|\leq 1$, and~$|d|< 1$}
    \}.
$$
Say that two elements of~$\fS_0$
are equivalent if, beyond a certain point dependent on the two
elements, they agree on all tiles of $\fs$.  Let $\fS$~be
the space of equivalence classes.  Since $G$~acts transitively
on~$\fS_0$, a fortiori it acts transitively on~$\fS$.  Thus
$\fS=G/H$ for some~$H$. In fact
$$
H=\{
    \begin{pmatrix} a_1 & b_2 & b_3 \\ 0 & a_2 & c \\ 0 & d & a_3
    \end{pmatrix} \in G ;
       \text{~$|a_j|=1$,~$|c|\leq 1$, and~$|d|< 1$}
  \}.
$$
The only relevant facts about~$H$ are that it is closed and
noncompact.  On~$\fS_0$ and~$\fS$ we put the topologies and
measures obtained from the isomorphisms with~$G/H_0$ and~$G/H$.

The Howe--Moore Theorem \cite[Theorems 10.1.4 and 2.2.6]{z} shows that
$\G$~acts \emph{ ergodically} on~$\fS$.  Suppose that there exist $a,b
\in A$ which cannot occur as $a=\G o(s)$ and $b=\G t(s)$ with $s\in
\fS_k$, for any $k$.  Let $\fS_0' = \{s\in\fS_0; \G o(s)=a\}$ and let
$\fS'$~be the projection of~$\fS_0'$ to~$\fS$.  The sets $\fS_0'$
and~$\fS'$ are clearly $\G$-invariant.  The set $\fS'$ is open, since
$\fS_0'$ is open.  Therefore $\fS'$ is not of null measure.  By the
ergodicity of the $\G$-action $\fS'$ has full measure.  Also of full
measure will be the inverse image of $\fS'$ in~$\fS_0$, which consists
of all those $s'\in\fS_0$ which are equivalent to some $s\in\fS_0$
with $\G o(s)=a$.  A fortiori
\begin{equation*}
\left\{
   s\in\fS_0 ; \text{there exists $K\ge 0$ such that $b\ne \G s_k$ for all $k \ge K$}
\right\}
\end{equation*}
is of full measure in~$\fS_0$.
Now, in~$\Gamma \backslash \fS_0$ the set
\begin{equation*}
\left\{
   \G s\in\G \backslash \fS_0 ; \text{there exists $K\ge 0$ such that $b\ne \G s_k$ for all $k \ge K$}
\right\}
\end{equation*}
will also be of full measure.

On~$\G \backslash \fS_0$ use the measure obtained in the usual
way from the unique (up to positive constant) positive $G$-invariant
measure on~$\fS_0$.  The following condition defines the new
measure, $d\dot{s}$, in terms of the old measure, $ds$:
$$
\int_{\Gamma \backslash \fS_0}
   \sum_{\gamma\in\Gamma}
      F(\gamma(s))
         \, d\dot{s}
   =
\int_{\fS_0}
   F(s)
      \, ds
$$
for any~$F\in C_c(\fS_0)$.  Since $\Gamma \backslash G$~is of
finite total measure, it follows that $\Gamma \backslash \fS_0$~is
too. Assume that this total measure is one.  One can easily  verify that relative to~$d\dot{s}$ the
distribution of $\G s_k$ is independent
of~$k$.  In fact, $\left| \{\G s \in \G \backslash \fS_0 ; \G s_k= b \}\right|= \frac{1}{\#A}$ for all $k \ge 0$.

 Let $0 < \epsilon < \frac{1}{\#A}$.  The monotone convergence theorem
implies that there exists~$K$ such that
$$
\left|\left\{
  \G s\in\G \backslash \fS_0 ; \text{$b\ne \G s_k$ for all $k \ge K$}
\right\}\right|
   >1-\epsilon.
$$
But this means that
$$
\left|\{\G s\in\G \backslash \fS_0 ; \G s_K=b \}\right|
   \le \epsilon.
$$
This contradicts $\left| \{\G s \in \G \backslash \fS_0 ; \G s_K= b \}\right|= \frac{1}{\#A}$,
and so proves the result.
\end{proof}

\begin{corollary}\label{H(3)+}
Let $\cB$ be the building of $G=\PGL_3(\KK)$, where $\KK$ is a local field of characteristic zero.
Let $\G$ be a lattice in $\PGL_3(\KK)$ which acts freely on the vertices of $\cB$.
Then the conclusions of Theorem~\ref{H(3)} hold.
\end{corollary}

\begin{proof}
The image of~$SL_3(\KK)$ in~$PGL_3(\KK)$ has finite index.  Let $\Gamma'$~be
the pullback to~$SL_3(\KK)$ of $\Gamma$.  Then
$\Gamma' \backslash SL_3(\KK)$ also has finite volume and the proof of
Theorem~\ref{H(3)} applies.  Moreover, the $\Gamma$-orbits of tiles
of~$\cB$ are made up of unions of $\Gamma'$-orbits.  So if we wish
to construct $s \in \fS_k$
having first and last tiles in certain $\Gamma$-orbits, we just
pick $\Gamma'$-orbits contained in the two $\Gamma$-orbits and
thereafter work with~$\Gamma'$.
\end{proof}

\begin{remark} We needed to use an indirect argument in the previous Corollary
because the Howe--Moore theorem does not apply in its simplest form to~$PGL_3(\KK)$.
\end{remark}
\begin{remark}
In work which will appear elsewhere, it will be shown how to extend
the methods of the proof of the Howe--Moore theorem so as to prove the
necessary ergodicity in greater generality.  It is enough to suppose
that $\G$~acts freely and with finitely many orbits on the vertices of
a thick building of type~$\wt A_2$.  Since ergodicity implies~(H2),
and since (H0), (H1), and~(H3) always hold, Theorem~\ref{main2} is
likewise true in this generality.

Not only does this allow one to work with the buildings associated to
$\PGL_3(\KK)$ when~$\KK$ has positive characteristic, but it also
makes available those buildings of type~$\wt A_2$ which are associated
to no linear group.  Note finally that direct combinatorial proofs
of~(H2) can be constructed for the $\wt A_2$~groups listed
in~\cite{cmsz}.

\end{remark}

\subsection{Examples of type~$\wt A_1 \times \wt A_1$}\label{A1A1}

Analogous results hold for groups acting on buildings of type
$\wt A_1 \times \wt A_1$.  Consider by way of
illustration a specific example studied in \cite{moz} and generalized
in \cite{bm}.  In \cite[Section 3]{moz}, there is constructed a
certain lattice subgroup $\G$ of $G = PGL_2(\QQ_p) \times
PGL_2(\QQ_q)$, where $p,q \equiv 1$ (mod 4) are two distinct
primes. The building $\D$ of $G$ is a product of two homogeneous trees
$T_1$, $T_2$, so that the chambers of $\D$ are squares and the
apartments are copies of the euclidean plane tesselated by squares.
If $a \in \G$ then there are automorphisms $a_1$, $a_2$ of $T_1$,
$T_2$, respectively such that $a(u,v)=(a_1 u,a_2 v)$ for each vertex
$(u,v)$ of $\D$.  However, even though each~$a\in\G$ is a direct
product of automorphisms, the group~$\G$ is not a direct product of
groups $\G_1$ and $\G_2$.

The group $\G$ acts freely and transitively on the vertices of $\D$.
The preceding results all extend to this situation. The tiles are now
squares instead of parallelograms. The boundary of $\D$ is defined as
before, using $\frac{\pi}{2}$-angled sectors.  The condition~(H1) is a
consequence of \cite [Theorem 3.2]{moz} and the irreducibility
condition~(H2) follows from \cite [Proposition 3]{moz'}.

\end{document}